\documentclass[reqno]{amsart}
\usepackage{amssymb, amsmath, amsthm, amscd, mathrsfs}

\usepackage{paralist}
 
\usepackage{bigints}
\usepackage{empheq}
\usepackage{amssymb}
\usepackage{cases}
\usepackage{enumitem}
\allowdisplaybreaks

\usepackage[colorlinks]{hyperref}
\usepackage[colorinlistoftodos]{todonotes}

\theoremstyle{plain}
\newtheorem{theorem}{Theorem}[section]
\newtheorem{corollary}{Corollary}

\newtheorem{lemma}[theorem]{Lemma}
\newtheorem{proposition}{Proposition}

\theoremstyle{definition}

\newtheorem{remark}{Remark}

\def \dv {\mathrm{div}}

\title[The cost of approximate controllability]
      {The cost of approximate controllability of heat equation with general dynamical bounadary conditions}

\author{I. Boutaayamou}
\address{I. Boutaayamou, Ibn Zohr University, Polydisplinary faculty, Ouarzazate,  B.P. 638, 45000, Morocco}
\email{dsboutaayamou@gmail.com}	

\author{S. E. Chorfi}
\address{S. E. Chorfi, Cadi Ayyad University, Faculty of Sciences Semlalia, 2390, Marrakesh, Morocco}
\email{chorphi@gmail.com}	

\author{L. Maniar}
\address{L. Maniar, Cadi Ayyad University, Faculty of Sciences Semlalia, 2390, Marrakesh, Morocco}
\email{maniar@uca.ma}

\author{O. Oukdach}
\address{O. Oukdach, Cadi Ayyad University, Faculty of Sciences Semlalia, 2390, Marrakesh, Morocco}
\email{omar.oukdach@gmail.com}

\begin{document}
	\maketitle 
\begin{abstract}
	We consider the heat equation with dynamic bounary conditions  involving gradient terms in a bounded domain. In this paper we  study the cost  of approximate controllability for this equation. Combining new developed  Carleman
estimates and some optimization techniques,  we obtain explicit bounds of the minimal norm control. We consider the linear and the semilinear cases.
\end{abstract}
\section{Introduction}
Consider the heat equation with dynamic boundary conditions, more precisely we deal with the following system 
{\small
	\begin{equation}
	\left\{\begin{array}{ll}
	{\partial_t y-\text{div}(\mathcal{A}\nabla y)  +B(x,t)\cdot \nabla y+ a(x,t)y= 1_{\omega}v }& {\text { in } \Omega\times(0,T),} \\
	{\partial_t y_{_{\Gamma}}   - \text{div}_{\Gamma}(\mathcal{A}_{\Gamma}\nabla_{\Gamma} y_{\Gamma})    + \partial_{\nu}^{\mathcal{A}}y+B_{\Gamma}(x, t)\cdot \nabla_{\Gamma} y_{\Gamma} + b(x,t)y_{\Gamma}=0} &\,\,{\text {on} \,\Gamma\times(0,T),} \\
	y_{\Gamma}(x,t) = y\rvert_{\Gamma}(x,t) &\,\,\text{on } \Gamma\times(0,T), \\
	{(y(0),y_{\Gamma}(0)) =(y_0,y_{\Gamma,0})} & {\text { in } \Omega \times \Gamma} \label{1.1}.
	\end{array}\right.
	\end{equation}}
Here, $T>0$ is  a fixed control time, $\Omega$ a bounded open set of $\mathbb{R}^N$, $N \geq 2$, with smooth boundary $\Gamma :=\partial\Omega$, and the control region $\omega$ is an arbitrary nonempty open subset which is strictly contained in $\Omega$ i.e., $\omega\Subset\Omega$. To abridge the notations, we denote in the sequel  $\Omega_T:=\Omega\times(0, T)$, $\Gamma_T:=\Gamma\times(0, T)$ and $\omega_T:=\omega\times(0, T)$. 
By $y_{\Gamma}$, one denotes the trace of $y$, that is, $y_{\Gamma}=y_{|_{\Gamma}}$, and $y_0 \in L^2(\Omega)$ the initial data
in the bulk and $y_{\Gamma, 0} \in  L^2(\Gamma)$ on the
boundary. We emphasize that $y_{\Gamma,0}$
is not necessarily the trace of $y_{0}$, since we do not assume that $y_0$ has a trace, but if $y_0$ has a
well-defined trace on $\Gamma$, then the trace must coincide with  $y_{\Gamma, 0}$. 
$\nu$ is the outer unit normal field, $  \partial^{\mathcal{A}}_{\nu}y:= (\mathcal{A}\nabla y\cdot\nu){\arrowvert_\Gamma}$ is the
co-normal derivative at $\Gamma$.
The functions   $a$, $b$, $B$ and $B_\Gamma$ are assumed belong to   $L^{\infty}(\Omega_T)$, $L^{\infty}(\Gamma_T)$, $L^{\infty}(\Omega)^N$ and  $L^{\infty}(\Gamma_T)^N$, respectively;  $1_\omega$  is the  characteristic function of $\omega$ and $v$ is the control function which acts on the system through the subset $\omega$. 
For the   divergence operator $\dv(\cdot)$ and tangential divergence operator $\dv_{\Gamma}(\cdot)$, we recall from \cite{KhMaMaGh19}, the following definitions.
For $F\in (L^2(\Omega))^N$ and  $F_{\Gamma}\in (L^2(\Gamma))^N$ 
$$ \dv(F): H^1(\Omega) \longrightarrow \mathbb{R},\quad u\longmapsto -\int_{\Omega} F\cdot\nabla u\, dx + \langle F\cdot \nu,u_{|\Gamma}\rangle_{H^{-\frac{1}{2}}(\Gamma),H^{\frac{1}{2}}(\Gamma)},$$
$$ \dv_{\Gamma}(F_{\Gamma}): H^1(\Gamma) \longrightarrow \mathbb{R},\quad u_{\Gamma}\longmapsto -\int_{\Gamma} F_{\Gamma}\cdot\nabla _{\Gamma} u_{\Gamma}\, d\sigma.$$
Here, $\nabla _{\Gamma} u_{\Gamma}$ is the tangential gradient of $u_{\Gamma}$ on $\Gamma$. The boundary $\Gamma$ of $\Omega$ is considered to be a $(N-1)$-dimensional compact Riemannian
submanifold equipped by  the Riemannian metric $g,$  induced by
the natural embedding $\Gamma \subset \mathbb{R}^N$, and the natural surface measure $d\sigma.$ See for instance  the second section in \cite{LaTrYa99} or \cite{C2} for more details.
Viewed as linear forms on $H^1(\Omega)$ and
$H^1(\Gamma)$, $\dv(F)$ and  $\dv_{\Gamma}(F_{\Gamma})$ are   continuous. In particular, we have the following estimates
\begin{align*}
|\langle \dv(F), u\rangle| \leq C_1  \|F\|_{(L^2{(\Omega))^N}}\|u\|_{H^{1}(\Omega)}, \quad F\in (L^2(\Omega))^N, u\in H^{1}(\Omega),\\
|\langle \dv_{\Gamma}(F_{\Gamma}), u_{\Gamma}\rangle| \leq C_2 \|F_{\Gamma}\|_{(L^2{(\Gamma))^N}}\|u_{\Gamma}\|_{H^{1}(\Gamma)}, \quad F_{\Gamma}\in (L^2(\Gamma))^N, u_{\Gamma}\in H^{1}(\Gamma)
\end{align*}
for some positive constants $C_1$ and $C_2$,  see {\cite[Proposition 3.6]{KhMaMaGh19}} for more details. 
For $F_{\Gamma} =\nabla_{\Gamma}u_{\Gamma}, u_{\Gamma}\in H^1(\Gamma)$, we define the Laplace-Beltrami operator $\Delta_\Gamma$ as follows 
\begin{equation*}
\Delta_\Gamma u_{\Gamma}= \dv_{\Gamma}(\nabla_{\Gamma}u_{\Gamma}).
\end{equation*}
Throughout this paper, we assume that the matrices $\mathcal{A}$ and $\mathcal{A}_{\Gamma}$ satisfy 

Throughout this paper, we assume that the matrices $\mathcal{A}$ and $\mathcal{A}_{\Gamma}$ satisfy 
\begin{enumerate}
	\item [(i)] $\mathcal{A}(\cdot)=(c(\cdot)_{i,j})\in C^1(\overline{\Omega}, \mathbb{R}^{N\times N})$ and $\mathcal{A}_{\Gamma}(\cdot)=(c_{_{\Gamma}}(\cdot)_{i,j})\in C^1(\Gamma, \mathbb{R}^{N\times N})$ are  symmetric, i.e., $c(x)_{i,j}=c(x)_{j,i}$ and $c_{_{\Gamma}}(x)_{i,j}=c_{_{\Gamma}}(x)_{j,i}$, 
	\item [(ii)]  $\mathcal{A}(\cdot)$ and $\mathcal{A}_{\Gamma}(\cdot)$ are uniformly elliptic, in particular there exist constants  $\alpha>0 $ and $\alpha_{_{\Gamma}}>0 $  such that 
	\begin{equation}\label{el1.2}
	\langle \mathcal{A}(x)\zeta, \zeta\rangle\geq	\alpha|\zeta|^2, \quad \text{and}\,\,\langle \mathcal{A}_{\Gamma}(x_{_{\Gamma}})\zeta, \zeta\rangle\geq	\alpha_{_{\Gamma}}|\zeta|^2
	\end{equation}
	for each $x\in \overline{\Omega}$, $x_{_{\Gamma}}\in \Gamma$, and  $\zeta\in\mathbb{R}^N.$
\end{enumerate}

Several authors have studied the existence, uniqueness and the regularity of solutions to the system \eqref{1.1}. We refer for example to the papers \cite{MiZi05,KhMaMaGh19, B6,B8, B10, C2, B11} and the references therein. This type of boundary conditions has been also called generalized Wentzell or generalized Wentzell-Robin boundary conditions and arises for many known equations of mathematical physics and are  motivated by, among others, problems in diffusion phenomena, reaction-diffusion systems in phase-transition phenomena, special flows in hydrodynamics, models in climatology, and so on. We refer to \cite{Gol06} for a derivation and physical interpretation of such boundary conditions.

More recently, the controllability question of such kind of problems  has been also investigated, especially in papers \cite{B8}, \cite{CiGa'17} and \cite{KhMa19}. In {\cite{B8}}, the authors have proved  interior and boundary  controllability results of (\ref{1.1}). More precisely, they proved that for all $\omega\Subset\Omega$ (resp. $\Gamma_0\subset \Gamma$),  there is an  internal control (resp. boundary control) such that the solution to (\ref{1.1}) with $\mathcal{A}=\mathcal{A}_{\Gamma}=I,$ $I$ the identity matrix, and $B=B_\Gamma=0,$  satisfies both the null and approximate controllability properties. In \cite{KhMa19} the authors studied the  same question in the  case $B\neq0$ and $B_{\Gamma}\neq0$, but with $\mathcal{A}=\mathcal{A}_{\Gamma}=I$. It is worth noting that the main key, in \cite{KhMa19} and \cite{B8},  to obtain  the above controllability results  is the  establishment  of suitable Carleman estimates. Note that null controllability of parabolic  and hyperbolic equations, using Carleman estimates, is extensively  studied in the case of static boundary conditions, see e.g. \cite{Ba02, LaTrYa97, Cor12, ImYa03, Ba00, Ba99, B3, B5, Cor07}. In \cite{CiGa'17},  the authors have dealt with the controllability  of wave equation with dynamic conditions.              

As well known, see for instance\cite{Cor12},  the  problem of approximate controllability for system \eqref{1.1}  can be reduced to the following unique continuation property: \begin{equation}\label{01.3}
\varphi=0 \quad \text{in} \,\,\omega_T \Longrightarrow (\varphi=0 \quad\text{in} \,\,\Omega_T \,\,\text{and}  \,\,\varphi_\Gamma=0 \quad\text{on} \,\, \Gamma_T),        
\end{equation} 
where $\Phi=(\varphi,\varphi_\Gamma)$ is the solution of the  so-called adjoint system
{\small
	\begin{equation}
	\left\{\begin{array}{ll}
	{-\partial_t \varphi-\text{div}(\mathcal{A}\nabla \varphi)  -\text{div}(\varphi B)+ a(x,t)\varphi=0}& \,{\text {in} \,\Omega_T,} \\
	{-\partial_t \varphi_{\Gamma}   -\text{div}_{\Gamma}(\mathcal{A}_{\Gamma}\nabla_{\Gamma} \varphi_{\Gamma})     + \partial^{\mathcal{A}}_{\nu}\varphi - \text{div}_{\Gamma}(\varphi_{\Gamma} B_{\Gamma})+\varphi_{\Gamma} B\cdot\nu + b(x,t)\varphi_{\Gamma}=0}&\,{\text{on}\,\Gamma_T,} \\
	\varphi_{\Gamma}(x,t)=\varphi\rvert_{\Gamma}(x,t)&\,\text{on}\,\Gamma_T, \\
	{(\varphi(T),\varphi_{\Gamma}(T) ) =(\varphi_T,\varphi_{\Gamma,T})} &\,{\text {in}\, \Omega \times\Gamma} \label{1.5}.
	\end{array}\right.
	\end{equation}  }  
Using a suitable Carleman estimate, Lemma \ref{l5.8} below,  one can show that  the property \eqref{01.3}    holds, and the   approximate  controllability of \eqref{1.1}   is then satisfied.

The main concern in the present paper is to obtain explicit bound of the cost
of  interior  approximate controllability, that is,  an explicit bound, with respect to data,  of the minimal norm of  controls needed to control the system \eqref{1.1} approximately.

In order to  clarify this purpose, let us recall that  it is easy,   by Carleman estimate below, to show that for each $\varepsilon>0$ and $Y_1=(y_1, y_{\Gamma,1})\in\mathbb{L}^2:= L^2(\Omega)\times L^2(\Gamma)$ there exists a control $v\in L^2(\omega_T)$ such that
the solution  $Y= (y, y_\Gamma)$ to (\ref{1.1}) satisfies 
\begin{equation}\label{1.3}
\|y(T)-y_1\|_{L^2(\Omega)}\,+\, \|y_{\Gamma}(T)-y_{\Gamma,1}\|_{L^2(\Gamma)}\leq \varepsilon,
\end{equation}
where, in (\ref{1.3}), $Y(T)= (y(T), y_\Gamma(T))$ is the solution to (\ref{1.1}) at $t=T$. In fact, a null controllability result can be  proved.
In particular, inequality (\ref{1.3}) implies  that the set of  admissible internal controls \, $\mathcal{U}_{ad} (Y_0,Y_1,\varepsilon,\omega)$
$$\mathcal{U}_{ad} (Y_0,Y_1,\varepsilon,\omega)=\big\{ v\in L^2(\omega_T) :   \|y(T)-y_1\|_{L^2(\Omega)}\,+\, \|y_{\Gamma}(T)-y_{\Gamma,1}\|_{L^2(\Gamma)}\leq \varepsilon\big\}
$$
is nonempty.

\par       Let us introduce the following quantity, which measures  the cost of  interior  approximate controllability, or, more precisely, the cost of achieving \eqref{1.3}, namely
$$ \mathcal{C}(Y_0,Y_1,\varepsilon, \omega)= \inf_{v\in \mathcal{U}_{ad} (Y_0,Y_1,\varepsilon,\omega)}\|v\|_{L^2(\omega_T)}.$$
The  main aim of this paper is to obtain explicit bound of $\mathcal{C}(Y_0,Y_1,\varepsilon, \omega)$, and to show how to use this result to deal with the controllability issue for some semilinear problems,  see Section \ref{sss.3} below.

\par Observe that, taking into account that systems (\ref{1.1}) is linear, one can assume, without loss of generality, that $Y_0=0$. Indeed,
$$\mathcal{C}(Y_0,Y_1,\varepsilon,\omega)= \mathcal{C}(0,Z_1,\varepsilon,\omega),$$
where $Z_1= Y_1- Y(T,Y_0, v=0)$, and  $Y(T,Y_0, v=0)$ is the solution of (\ref{1.1}) with $v=0$ at $t=T$.

Several motivations can be found for this cost of  controllbility issue. Let us recall that  parabolic systems, due to their regularizing effect, are, roughly speaking,  approximately but not exactly controllable,  it is then  natural to study the cost of their approximate controllability, or,
in other words, the size of the control needed to reach to a neighborhood
of a final state which is not exactly reachable. This problem has been the object of intensive research  in the past few years.   Thus, it has been analyzed in several recent papers. Among them, let us mention  \cite{ BE'16, B4, Mi04, B7, B12}. In \cite{B4} and \cite{B7},  the authors obtained explicit bound of the cost of approximate controllability for linear heat equation, while \cite{B12} and  \cite{BE'16} treated the same question for semilinear heat equations   and for  Ginzburg-Landau
equation respectively. In \cite{LeLiGa09}, some results 
on $L^p$ and $L^{\infty}$-type cost estimates are obtained for the static heat equation. The paper \cite{Mi04} addressed the following question: How does the geometry of the control region influence the cost of controlling the heat to zero in small time? This problem of control cost,  as pointed out by several authors, is also important in engineering cybernetics.

For our further results, we  remind the following fundamental Fenchel-Reckafellar duality theorem.
\begin{theorem}\label{RF.1}$($\cite{Ek'74}$)$
	Let $X$, $Y$ be two Hilbert spaces  and $L\in\mathcal{L}(X,Y)$ a linear continuous operator. Let $F:X\longrightarrow \mathbb{R}\cup\{+\infty\}$ and    $G:Y\longrightarrow \mathbb{R}\cup\{+\infty\}$ be two convex and lower semi-continuous functions, and  assume  $0 \in \text{int}(L(D(F)-D(G) ))$.  Then,   we have
	\begin{equation*}
	\inf_{x\in X}[F(x) + G(L x)] = -\inf_{y\in Y}[F^{\star}(L^{\star}y) + G^{\star}(-y)],
	\end{equation*} 
	where $L^{\star}$ is the adjoint operator of $L$,  $F^{\star}$ and $G^{\star}$ are the conjugate functions of $F$ and $G$, respectively.
\end{theorem}

The paper is organized as follows. In Section \ref{WP0.1}  we  introduce some functional spaces which allow us to give sense to the solution of the considered equations, and we recall some previous results on the well-posedness issue.  In Section \ref{Caar0.2}, we state and show several Carleman estmates; needed for our aim.    We consider the linear case, in   Section \ref{subs.1},  and  the semilinear case in  Section \ref{sss.3}.
\section{Functional setting and well-posedness }\label{WP0.1}
Let $\Omega\subset \mathbb{R}^N$ be a bounded open set with smooth boundary $\Gamma :=\partial\Omega$ and $1 \leq p\leq \infty$. Following \cite{B8}, we introduce the product space defined by
$$ \mathbb{L}^p = L^p(\Omega)\times L^p(\Gamma), \qquad 1 \leq p\leq \infty.$$
Here,  we have considered the Lebesgue measure $dx$ on $\Omega$ and the natural surface measure $d\sigma$ on $\Gamma.$
Equipped by the norm
$$  \|(u, u_\Gamma)\|_{\mathbb{L}^p}= \big( \|u\|^p_{L^{p}(\Omega)} + \|u_\Gamma\|^p_{L^p(\Gamma)}   \big )^{\frac{1}{p}},\quad 1\leq p < \infty,$$
$\mathbb{L}^p$ is a Banach space.
For $p=\infty$, we set {\small$\|(u, u_\Gamma)\|_{\mathbb{L}^{\infty}} = \max\{\|u\|_{L^{\infty}(\Omega)}, \|u_\Gamma\|_{L^{\infty}(\Gamma)} \},$}  and we have $(\mathbb{L}^{\infty}, \|.\|_{\mathbb{L}^{\infty}})$ is a Banach space.   
Observe that  $\mathbb{L}^p$ can be identified with the space $L^p(\overline{\Omega}, d\mu)$, where the measure $\mu$ is defined on $\overline{\Omega}$,    for every measurable set $ B\subset \overline{\Omega}$,
by
$$  \mu(B)= \int_{B\cap\Omega}dx \,+\sigma(B\cap \Gamma).$$
In addition, it is  known  that $ \mathbb{L}^2 $ is a real Hilbert space with the scalar product 
$$ \big\langle (u,w),(v,z)\big\rangle_{\mathbb{L}^2}= \langle u,v\rangle_{L^2(\Omega)} + \langle w,z\rangle_{L^2(\Gamma)} = \int_{\Omega}u v \,dx +  \int_{\Gamma}w z \,d\sigma.$$
Recall that $H^1(\Gamma)$ and $H^2(\Gamma)$ are real Hilbert spaces endowed with the respective norms
$$ \|u\|_{H^1(\Gamma)}= \langle u,u \rangle^{\frac{	1}{2}}_{H^1(\Gamma)}, \,\text{with} \,\,\langle u,v \rangle_{H^1(\Gamma)}= \int_{\Gamma}u v d\sigma + \int_{\Gamma}\nabla_{\Gamma} u \nabla_{\Gamma}v d\sigma, \quad $$
and 
$$ \|u\|_{H^2(\Gamma)}= \langle u,u \rangle^{\frac{	1}{2}}_{H^2(\Gamma)},\, \text{with}\,\, \langle u,v \rangle_{H^2(\Gamma)}= \int_{\Gamma}u v \,d\sigma +\int_{\Gamma}\Delta_{\Gamma} u \Delta_{\Gamma}v d\sigma.\quad $$                 
We point out that the operator $\Delta_{\Gamma}$  can be considered as an unbounded linear operator from $L^2(\Gamma)$ in $L^2(\Gamma)$, with domain
$$D(\Delta_{\Gamma}) =\{ u\in L^2(\Gamma) : \,\, \Delta_{\Gamma}u\in L^2(\Gamma)\},$$
and is known   that $-\Delta_{\Gamma}$ is  a self-adjoint and nonnegative operator on $L^2(\Gamma)$. This implies that  $-\Delta_{\Gamma}$ generates an analytic  $C_0$-semigroup $(e^{t\Delta_\Gamma})_{t\geq 0}$  on $L^2(\Gamma)$. If $\Gamma$ is
smooth, then one can show that $D(\Delta_\Gamma) = H^2(\Gamma)$, and $u\mapsto \|u\|_{L^2(\Gamma)}+ \|\Delta_{\Gamma}u\|_{L^2(\Gamma)}$ defines an equivalent norm on $H^2(\Gamma)$,  see {\cite{B8, B10,B11}} and    the references therein for more details.  As in {\cite{B8}}, we denote 
$$\mathbb{H}^k=\{(u,u_\Gamma)\in H^k(\Omega)\times H^k(\Gamma): \,\, u_\Gamma=u_{|\Gamma}\},\quad k=1,2,$$
viewed as a subspace of $H^k(\Omega)\times H^k(\Gamma)$ with the natural topology inherited by $H^k(\Omega)\times H^k(\Gamma)$,  where $u_{|\Gamma}$  denotes the trace of $u$ on $\Gamma$,
and
$$ \mathbb{E}(t_0, t_1)= H^2(t_0,t_1; \mathbb{L}^2) \cap L^2(t_0,t_1; \mathbb{H}^2), \quad \text{for} \,\,t_1>t_0,  \, \text{and} \,\,\mathbb{E}_1=\mathbb{E}(0, T).$$
Notices that the trace theorem show that the space $\mathbb{H}^1$ is closed in $H^1(\Omega)\times H^1(\Gamma)$, and 
the norm
$$ \|(u,u_\Gamma)\|_{\mathbb{H}^1}= \langle (u,u_\Gamma),(u,u_\Gamma) \rangle^{\frac{1}{2}}_{{\mathbb{H}^1}},$$ $$ \text{where}\quad  \langle (u,u_\Gamma),(v,v_\Gamma) \rangle_{{\mathbb{H}^1}} = \int_{\Omega}\nabla u \nabla v\, dx \,+\, \int_{\Gamma}\nabla_{\Gamma} u_\Gamma \nabla_{\Gamma} v_\Gamma\, d\sigma \,+\,\int_{\Gamma} u_\Gamma  v_\Gamma\, d\sigma,\qquad $$
is equivalent in $\mathbb{H}^1$ to the standard norm inherited by $H^1(\Omega)\times H^1(\Gamma).$
Similarly,  for  $\mathbb{H}^2$, the norm    
$$ \|(u,u_\Gamma)\|_{\mathbb{H}^2}= \langle (u,u_\Gamma),(u,u_\Gamma) \rangle^{\frac{1}{2}}_{{\mathbb{H}^2}},$$ $$ \text{where}\quad  \langle (u,u_\Gamma),(v,v_\Gamma) \rangle_{{\mathbb{H}^2}} = \int_{\Omega}\Delta u \Delta v\, dx \,+\, \int_{\Gamma}\Delta_{\Gamma} u_\Gamma \Delta_{\Gamma} v_\Gamma\, d\sigma \,+\,\int_{\Gamma} u_\Gamma  v_\Gamma\, d\sigma,\qquad $$
is equivalent in $\mathbb{H}^2$ to the standard norm inherited by $H^2(\Omega)\times H^2(\Gamma).$ 

We denote by $(H^k(\Omega))'$, $H^{-k}(\Gamma)$ and $\mathbb{H}^{-k}$, the dual of $H^k(\Omega)$, $H^{k}(\Gamma)$ and  $\mathbb{H}^{k}$, respectively,  $k=1,2$, and  
$$	\mathbb{W}=\{U\in L^2(0,T; \mathbb{H}^1): \,\,\, U^{'}\in L^2(0,T; \mathbb{H}^{-1}) \}.$$
Recall also that
$\mathbb{H}^1$ (resp. 	$\mathbb{E}_1$) embeds compactly into $\mathbb{L}^2$ (resp. $L^2(0,T; \mathbb{H}^1))$. This will be essential when dealing with the semilinear case in Section \ref{sss.3} below.

Now we shall recall some results  on  the  well-posedness of the nonhomogeneous  forward    system
{\small
	\begin{equation}
	\left\{\begin{array}{ll}
	{\partial_t y-\text{div}(\mathcal{A}\nabla y)  +B(x,t)\cdot \nabla y+ a(x,t)y= f }& {\text { in } \Omega_T,} \\
	{\partial_t y_{_{\Gamma}}   - \text{div}_{\Gamma}(\mathcal{A}_{\Gamma}\nabla_{\Gamma} y_{\Gamma})    + \partial_{\nu}^{\mathcal{A}}y+B_{\Gamma}(x, t)\cdot \nabla_{\Gamma} y_{\Gamma} + b(x,t)y_{\Gamma(0)}=g} &\,\,{\text {on} \,\Gamma_T,} \\
	y_{\Gamma}(x,t) = y\rvert_{\Gamma}(x,t) &\,\,\text{on } \Gamma_T, \\
	{(y(0),y_{\Gamma}(0)) =(y_0,y_{\Gamma,0})} & {\text { in } \Omega_T} \label{sys2.2}.
	\end{array}\right.
	\end{equation}}
and the non-homogeneous backward one
{\small
	\begin{equation}
	\left\{\begin{array}{ll}
	{-\partial_t \varphi-\text{div}(\mathcal{A}\nabla \varphi)  -\text{div}(\varphi B)+ a(x,t)\varphi=f_1}& {\text {in} \,\Omega_T,} \\
	{-\partial_t \varphi_{\Gamma}   -\text{div}_{\Gamma}(\mathcal{A}_{\Gamma}\nabla_{\Gamma} \varphi_{\Gamma})     + \partial^{\mathcal{A}}_{\nu}\varphi - \text{div}_{\Gamma}(\varphi_{\Gamma} B_{\Gamma})+\varphi_{\Gamma} B\cdot\nu + b(x,t)\varphi_{\Gamma}=g_1}&\,{\text{on}\,\Gamma_T,} \\
	\varphi_{\Gamma}(t,x)=\varphi\rvert_{\Gamma}(t,x)&\,\text{on}\,\Gamma_T, \\
	{(\varphi(T),\varphi_{\Gamma}(T) ) =(\varphi_T,\varphi_{\Gamma,T})} &\,{\text {in}\, \Omega \times\Gamma} \label{systadjoint2.3}.
	\end{array}\right.
	\end{equation}   }
Remark first that the system \eqref{sys2.2} can be rewritten as the following
abstract Cauchy problem
\begin{equation}\label{1.8}
\left\{
\begin{array}{ll}
Y'(t) = A\,Y - D(t)\,Y+F , \,   \,\,\, t>0,  \\

Y(0)=Y_0= (y_0,y_{\Gamma,0}),
\end{array}
\right.
\end{equation}
where, 
$$ Y=(y, y_\Gamma),\,\, F=(f, g),\quad A=\begin{pmatrix} 
\text{div}(\mathcal{A}\nabla )  & 0 \\
-\partial^{\mathcal{A}}_{\nu} & \text{div}_{\Gamma}(\mathcal{A}_{\Gamma}\nabla_{\Gamma}) 
\end{pmatrix}, \mathcal{D}(A) =\mathbb{H}^2$$
and 
$$D(t)=\begin{pmatrix} 
B(t)\cdot\nabla + a(t) & 0 \\
0 & B_{\Gamma}(t)\cdot\nabla_{\Gamma} + b(t)
\end{pmatrix}.$$


Following \cite{B8}, we can show that the operator A satisfies the following important property.
\begin{proposition}[\cite{B8}]\label{t4.2}
	The operator $A$ is densely defined, and generates an analytic $C_0$-semigroup $(e^{tA})_{t\geq 0}$ on $\mathbb{L}^2.$ We have also  $(\mathbb{L}^2, \mathbb{H}^2 )_{\frac{1}{2}, 2}= \mathbb{H}^1$.
\end{proposition}
The following existence and uniqueness results hold.
\begin{proposition} [\cite{Schn01}]\label{t4.22}
	For every $Y_0=(y_0,y_{\Gamma,0})\in\mathbb{L}^2$, $f\in L^2(\Omega_T)$ and  $g\in L^2(\Gamma_T)$,  the system \eqref{sys2.2} has a unique mild solution given by \begin{equation*}\label{4.2}
	Y(t)= e^{tA}Y_0 \,+\, \int_{0}^{t}e^{(t-s)A}(F(s)- D(s)Y(s))ds
	\end{equation*}
	for all $t\in [0,T]$.
	Moreover,  there exists a
	constant $C > 0$ such that
	\begin{equation}\label{4.3}
	\|Y\|_{C([0,T];\mathbb{L}^2)}\leq C\big( \|Y_0\|_{\mathbb{L}^2} + \|f\|_{L^2(\Omega_T) } + \|g\|_{L^2(\Gamma_T) } \big).
	\end{equation}
	
	
	\begin{proof}
		We have \,	$(\mathbb{L}^2, \mathbb{H}^2 )_{ \frac{1}{2}, 2}= \mathbb{H}^1$	and since $a, b, B,$ and  $B_{\Gamma}$ are bounded, then $D(t)\in \mathcal{L}(\mathbb{H}^1; \mathbb{L}^2)$, for all $t\in (0,T)$. It suffices then to apply Theorem 3.1 in \cite{Schn01}.\qed
	\end{proof}
\end{proposition}
For the backward system, we have the following well-posedness result,  see \cite{KhMaMaGh19} for the proof and more details.
\begin{proposition}\label{prop2.3}
	For every $\Phi_T=(\varphi_T,\varphi_{\Gamma,T})\in\mathbb{L}^2$ and $F_1=(f_1,g_1)\in$ $\\L^2(0,T; \mathbb{H}^{-1})$, the backward system \eqref{systadjoint2.3} has a unique weak solution {\small$\Phi=(\varphi, \varphi_{\Gamma)}\in \mathbb{W}^1$,} i.e., 
	\begin{align}
	&\int_{0}^{T}\langle \partial_t\varphi, v\rangle_{(H^1(\Omega))^{'}, H^1(\Omega) }\,dt + \int_{\Omega_T}\mathcal{A}\nabla\varphi\cdot\nabla v\, dx\,dt 
	- \int_{\Omega_T}\varphi B\cdot\nabla v\, dx\,dt\nonumber\\
	&+\int_{\Omega_T}a \varphi v \,dx\,dt + \int_{0}^{T}\langle \partial_t\varphi_{\Gamma}, v_{\Gamma}\rangle_{H^{-1}(\Gamma), H^1(\Gamma) }\,dt +\int_{\Gamma_T}\mathcal{A}_{\Gamma}\nabla_{\Gamma}\varphi_{\Gamma}\cdot\nabla_{\Gamma} v_{\Gamma} \,d\sigma\,dt \nonumber\\
	& -\int_{\Gamma_T}\varphi_{\Gamma} B_{\Gamma}\cdot\nabla_{\Gamma} v_{\Gamma} d\sigma\,dt +\int_{\Gamma_T}b\varphi_{\Gamma} v_{\Gamma}\, d\sigma\,dt=  \int_{0}^{T}\langle f, v\rangle_{(H^1(\Omega))^{'}, H^1(\Omega) }\,dt \nonumber\\
	& +\int_{0}^{T}\langle g, v_{\Gamma}\rangle_{H^{-1}(\Gamma), H^1(\Gamma) }\,dt\label{Kh1.12}
	\end{align}
	for each $(v,v_{\Gamma})\in L^2(0,T; \mathbb{H}^1)$, with  $v(0)=v_{\Gamma}(0)=0$ and $(\varphi(T), \varphi_{\Gamma}(T) )=(\varphi_T, \varphi_{\Gamma, T}).$
	Moreover, we have the estimate
	\begin{equation}\label{est2.10}
	\max_{0\leq t\leq T}\|\Phi(t)\|^2_{\mathbb{L}^2}+\|\Phi\|^2_{L^2(0,T ; \mathbb{H}^1)}+\|\Phi^{'}\|^2_{L^2(0,T; \mathbb{H}^{-1})}\leq C\big(\|\Phi_T\|_{\mathbb{L}^2}+ \|F\|^2_{L^2(0,T; \mathbb{H}^{-1})})
	\end{equation}  
	for some positive constant $C$.
\end{proposition}
\section{ Carleman estimate}\label{Caar0.2}
\par The main aim in this section is to establish a suitable Carleman estimate to the following general adjoint system
{\small
	\begin{equation}
	\left\{\begin{array}{ll}
	{-\partial_t \varphi-\text{div}(\mathcal{A}\nabla \varphi)= F_0-  \text{div}(F)}& {\text {in} \,\Omega_T,} \\
	{-\partial_t \varphi_{\Gamma}   -\text{div}_{\Gamma}(\mathcal{A}_{\Gamma}\nabla_{\Gamma} \varphi_{\Gamma} )     + \partial^{\mathcal{A}}_{\nu}\varphi=F_{\Gamma,0} +F\cdot\nu- \text{div}_{\Gamma}(F_{\Gamma})}&\,{\text{on}\,\Gamma_T,} \\
	\varphi_{\Gamma}(x,t)=\varphi\rvert_{\Gamma}(x,t)&\,\text{on}\,\Gamma_T, \\
	{(\varphi(T),\varphi_{\Gamma}(T) ) =(\varphi_T,\varphi_{\Gamma,T})} &\,{\text {in}\, \Omega \times\Gamma} \label{NE1.5}
	\end{array}\right.
	\end{equation} }
for $F_0\in L^2(0,T; L^2(\Omega))$, $F\in L^2(0,T; L^2(\Omega)^N)$, $F_{\Gamma,0}\in L^2(0,T; L^2(\Gamma))$  $F_{\Gamma}\in L^2(0,T; L^2(\Gamma)^N)$ and $(\varphi_T,\varphi_{\Gamma,T})\in  \mathbb{L}^2.$

\par Let us start with    introducing the  following  well-known Morse function, see for instance \cite{B5} and \cite{Cor12}.
\begin{lemma}
	Let  $\omega$ be a small open set of $\Omega$.  There is a function $\eta_0\in C^2(\overline{\Omega})$ such that
	\begin{equation*}\label{00.1}
	\left\{
	\begin{array}{ll}
	\eta_0> 0  &\quad  \mathrm{in}\,\,  \Omega  \quad  \mathrm{and} \quad\eta_0=0 \quad \mathrm{in} \,\, \Gamma,\\
	|\nabla\eta_0|  \neq 0  &\quad \mathrm{in}   \,\,\, \overline{\Omega\backslash\omega}.
	
	\end{array}
	\right.     
	\end{equation*}
	Furthermore, the function $\eta_0$ holds the  following  properties 
	\begin{equation*}\label{eta3.26}
	|\nabla_{\Gamma}\eta_0|=0, \quad \partial_{\nu}\eta_0<-c, \quad \nabla\eta_0= \partial_{\nu}\eta_0 \nu\quad \text{on} \,\,\Gamma
	\end{equation*}
	for some constant $c>0$.
\end{lemma}

Introduce the following   classical weight functions 
$$\xi(x,t)=\frac{e^{{\lambda (m \|\eta_0\|_{\infty} + \eta_0(x)) }}}{t(T -t)}  \quad and \quad \alpha(x,t)=\frac{e^{{2\lambda m \|\eta_0\|_{\infty} }}- e^{{\lambda (m\|\eta_0\|_{\infty} + \eta_0(x))}}}{t(T -t)},$$
where $(x,t)\in \Omega_T$, $m>1$ and $\lambda \geq 1.$ 

\vspace{0.4cm}
\par By arguing as in \cite{B8}, the authors  proved  in \cite{ACMO20}  the following Carleman estimate.
\begin{lemma}\label{l5.8}
	There exist $\lambda_0\geq 1$, $s_0=s_0(T+T^2)\geq 1$ and $C=C(\omega,\Omega)>0$ such that 
	\begin{align}
	&s^3\lambda^4\int_{\Omega_T}\xi^3 e^{-2s\alpha}|\varphi|^2dx\,dt \,\,+\,\,
	s^3\lambda^4\int_{\Gamma_T}\xi^3 e^{-2s\alpha}|\varphi_\Gamma|^2d\sigma\,dt\nonumber\\
	&\leq C\Big( s^3\lambda^4\int_{\omega_T}\xi^3 e^{-2s\alpha}|\varphi|^2dx\,dt+\int_{\Omega_T}e^{-2s\alpha}|\partial_t\varphi \,+\,\mathrm{div}(\mathcal{A}\nabla \varphi) \,|^2dx\,dt\,\,	+\nonumber\\
	&\,\,\int_{\Gamma_T}e^{-2s\alpha}|\partial_t\varphi_{\Gamma} +\mathrm{div}_{\Gamma}(\mathcal{A}_{\Gamma}\nabla_{\Gamma}\varphi_{\Gamma}) -\partial^{\mathcal{A}}_\nu\varphi |^2  d\sigma\,dt\Big) \label{Carleman2.15}
	\end{align}  
	for all \, $\Phi=(\varphi,\varphi_\Gamma)\in \mathbb{E}_1$, $  \lambda\geq \lambda_0$ and  $s \geq s_0.$  
\end{lemma}	\vspace{0.1cm}
\par Following \cite{KhMa19}  and using  Lemma \ref{l5.8},  we deduce a Carleman estimate for \eqref{NE1.5}. 
\begin{theorem}\label{ThKh2.7}
	Let  $F_0\in L^2(0,T; L^2(\Omega))$, $F\in L^2(0,T; L^2(\Omega)^N)$, $F_{\Gamma,0}\in$ \\$L^2(0,T; L^2(\Gamma))$ and $F_{\Gamma}\in L^2(0,T; L^2(\Gamma)^N)$, and  $\Phi=(\varphi,\varphi_{\Gamma})\in$ the unique weak
	solution  to the system \eqref{NE1.5}. Then,  there exist constants $\lambda_1>1$, $s_1>1$ and a constant $C=C(\omega, \Omega)$ such that 
	{\small
		\begin{align}
		&s^3\lambda^4\int_{\Omega_T}\xi^3 e^{-2s\alpha}|\varphi|^2dx\,dt \,\,+\,\,
		s^3\lambda^3\int_{\Gamma_T}\xi^3 e^{-2s\alpha}|\varphi_\Gamma|^2d\sigma\,dt\nonumber\\
		&\leq C\Big( s^3\lambda^4\int_{\omega_T}\xi^3 e^{-2s\alpha}|\varphi|^2dx\,dt+\int_{\Omega_T}e^{-2s\alpha}\xi^2|F_0|^2 dx\,dt	+ s\lambda^2\int_{\Omega_T}e^{-2s\alpha}\xi^2\|F\|^2_{\mathbb{R}^N} dx\,dt\nonumber\\
		&\,\,\int_{\Gamma_T}e^{-2s\alpha}\xi^2|F_{\Gamma,0}|^2 d\sigma\,dt+ \,\,s\lambda^2\int_{\Gamma_T}e^{-2s\alpha}\xi^2\|F_{\Gamma}\|^2_{\mathbb{R}^N} d\sigma\,dt\Big)\label{Kh.24}
		\end{align}}
	for any $\lambda\geq \lambda_1$ and $s\geq s_1$. 
\end{theorem}
\par To prove Theorem \ref{ThKh2.7}, as in  \cite{KhMa19}, we introduce the following operators
\[
\begin{aligned}
L y &=\partial_{t} y- \mathrm{div}(\mathcal{A} \nabla y), \quad L^{*} y=-\partial_{t} y-\mathrm{div}(\mathcal{A} \nabla y), \\
L_{\Gamma} y_{\Gamma} &=\partial_{t} y_{\Gamma}-\mathrm{div}_{\Gamma}(\mathcal{A}_{\Gamma} \nabla_{\Gamma} y_{\Gamma}), \quad L_{\Gamma}^{*} y_{\Gamma}=-\partial_{t} y_{\Gamma}-\mathrm{div}_{\Gamma}(\mathcal{A}_{\Gamma} \nabla_{\Gamma} y_{\Gamma})
\end{aligned}
\]
for $Y=\left(y, y_{\Gamma}\right) \in$ $\mathbb{E}_{1}$, and for $\left(\varphi, \varphi_{\Gamma}\right)$ the unique weak solution to \eqref{NE1.5}, consider the variational problem 
\begin{align}
&\int_{\Omega_{T}} e^{-2 s \alpha} L^{*} y L^{*} v d x \,d t+\int_{\Gamma_{T}} e^{-2 s \alpha}\left(L_{\Gamma}^{*} y_{\Gamma}+ \partial^{\mathcal{A}}_{\nu} y\right)\left(L_{\Gamma}^{*} v_{\Gamma}+ \partial^{\mathcal{A}}_{\nu} v\right) d \sigma \,d t
+ \nonumber\\
&s^{3} \lambda^{4} \int_{\omega_{T}} e^{-2 s \alpha} \xi^{3} y v d x\, d t+\int_{\Omega} y(T, x) v(T, x) d x+\int_{\Gamma} y_{\Gamma}(T, x) v_{\Gamma}(T, x) d \sigma\nonumber\\
&=-s^{3} \lambda^{4} \int_{\Omega_{T}} e^{-2 s \alpha} \xi^{3} v \varphi d x \,d t-s^{3} \lambda^{3} \int_{\Gamma_{T}} e^{-2 s \alpha} \xi^{3} v_{\Gamma} \varphi_{\Gamma} d \sigma\, d t \label{Kh.25}
\end{align}
for all $V:=\left(v, v_{\Gamma}\right) \in \mathbb{E}_{1} .$ 
As in \cite{KhMa19}, we can prove that for any $\lambda\geq\lambda_0$ and any $s\geq s_0$, the variational problem \eqref{Kh.25}
possesses exactly one solution $U\in \mathbb{E}_1$, and  we have the following intermediate Carleman estimate.

\begin{lemma}
	Let $\lambda\geq\lambda_0$ and any $s\geq s_0$, and $U = (u, u_{\Gamma})$ be the unique solution to the variational problem \eqref{Kh.25}, and set
	\begin{equation*}
	z=-e^{-2 s \alpha} L^{*} u, z_{\Gamma}=-e^{-2 s \alpha}\left(L_{\Gamma}^{*} u_{\Gamma}+ \partial^{\mathcal{A}}_{\nu} u\right), \quad v=s^{3} \lambda^{4} e^{-2 s \alpha} \xi^{3} u.
	\end{equation*}
	
	Then, the following assertions hold.
	\begin{itemize}
		\item[(i)]$Z=\left(z, z_{\Gamma}\right)$ is the unique strong solution to the system
		\begin{equation}
		\left\{\begin{array}{ll}
		{\partial_{t} q-\mathrm{div}(\mathcal{A} \nabla q)= s^{3} \lambda^{4} e^{-2 s \alpha} \xi^{3} \varphi+v 1_{\omega} }& {\mathrm{in}\,\, \Omega_T,} \\
		{\partial_{t} q_{\Gamma}-\mathrm{div}_{\Gamma}(\mathcal{A}_{\Gamma} \nabla_{\Gamma} q_{\Gamma})+ \partial^{\mathcal{A}}_{\nu} q=s^{3} \lambda^{3} e^{-2 s \alpha} \xi^{3} \varphi_{\Gamma}} &{\mathrm{on} \,\,\Gamma_T,} \\
		{(q(0),q_{\Gamma}(0) ) =(0,0)} & {\mathrm{in} \,\,\Omega \times \Gamma} \label{Kh.28},\\
		{(q(T),q_{\Gamma}(T) ) =(0,0)} & {\mathrm{in} \,\,\Omega \times \Gamma.} 
		\end{array}\right.
		\end{equation}		
	\end{itemize}
	
	(ii) There exist $\tilde{s}=\tilde{s}(\Omega, \omega), \tilde{\lambda}=\tilde{\lambda}(\Omega, \omega) \geq 1$ and $C=C(\Omega, \omega)$ such that for all
	$s \geq \tilde{s}\left(T+T^{2}\right)$ and $\lambda \geq \tilde{\lambda}$
	\begin{align}
	&s^{-3} \lambda^{-4} \int_{\omega_{T}} e^{2 s \alpha} \xi^{-3}|v|^{2} \, dx\, dt+\int_{\Omega_{T}} e^{2 s \alpha}|z|^{2} d x d t+\int_{\Gamma_{T}} e^{2 s \alpha}\left|z_{\Gamma}\right|^{2} d\sigma\, dt \nonumber\\
	&+\quad s^{-2} \lambda^{-2} \int_{\Omega_{T}} e^{2 s \alpha} \xi^{-2}|\nabla z|^{2} d x\, d t+s^{-2} \lambda^{-2} \int_{\Gamma_{T}} e^{2 s \alpha} \xi^{-2}\left|\nabla_{\Gamma} z_{\Gamma}\right|^{2} d \sigma\, d t\nonumber\\
	&\leq C\left(s^{3} \lambda^{4} \int_{\Omega_{T}} e^{-2 s \alpha} \xi^{3}|\varphi|^{2} d x \,d t+s^{3} \lambda^{3} \int_{\Gamma_{T}} e^{-2 s \alpha} \xi^{3}\left|\varphi_{\Gamma}\right|^{2} d \sigma d t\right). \label{Kh02.27}
	\end{align}
\end{lemma}

\begin{proof}
	For the first point see \cite{KhMa19}. The proof of the second one will be summarized in the following three steps. In what follows, $C$ stands for a generic positive constant only  depending on $\Omega$  and $\omega$, whose value can change from line to line.
	\\
	{\bf{\text { Step } 1: \text { Estimate of the three first terms.}}}
	By the same ideas as in \cite{KhMa19}, it is easy to deduce the first estimate, namely
	\begin{align}
	&s^{-3} \lambda^{-4} \int_{\omega_{T}} e^{2 s \alpha} \xi^{-3}|v|^{2} d x d t+\int_{\Omega_{T}} e^{2 s \alpha}|z|^{2} d x d t+\int_{\Gamma_{T}} e^{2 s \alpha}\left|z_{\Gamma}\right|^{2} d \sigma d t \nonumber\\
	&\leq C\left(s^{3} \lambda^{4} \int_{\Omega_{T}} e^{-2 s \alpha} \xi^{3}|\varphi|^{2} d x d t+s^{3} \lambda^{3} \int_{\Gamma_{T}} e^{-2 s \alpha} \xi^{3}\left|\varphi_{\Gamma}\right|^{2} d \sigma d t\right)\label{Kh2.032}
	\end{align}
	for all $s\geq s_0$ and $\lambda\geq \lambda_0$\\
	{\bf{Step 2. Estimates of the first order interior term.}}	
	
	We multiply the first equation of \eqref{Kh.28} by $s^{-2} \lambda^{-2} e^{2 s \alpha} \xi^{-2} z$ and integrate by parts with respect to the space variable. So, we obtain
	{\small
		\begin{align}
		&s^{-2} \lambda^{-2} \int_{\Omega_{T}} e^{2 s \alpha} \xi^{-2} \partial_{t} z z d x d t+ s^{-2} \lambda^{-2} \int_{\Omega_{T}} e^{2 s \alpha} \xi^{-2}\nabla z\cdot\mathcal{A}\nabla z d x d t\nonumber\\
		&-2  s^{-1} \lambda^{-1} \int_{\Omega_{T}} e^{2 s \alpha} \xi^{-1} \nabla \eta_{0} \cdot \mathcal{A}\nabla z z d x d t-2  s^{-2} \lambda^{-1} \int_{\Omega_{T}} e^{2 s \alpha} \xi^{-2} \nabla \eta_{0} \cdot\mathcal{A}\nabla z z d x d t+\mathbf{B}_{1}\nonumber\\
		&=s \lambda^{2} \int_{\Omega_{T}} \xi \varphi z d x d t+s^{-2} \lambda^{-2} \int_{\omega_{T}} e^{2 s \alpha} \xi^{-2} v z d x d t\label{Kh2.33}
		\end{align}}
	with $\displaystyle\mathbf{B}_{1}:=-s^{-2} \lambda^{-2} \int_{\Gamma_{T}} e^{2 s \alpha} \xi^{-2}z\,\partial^{\mathcal{A}}_{\nu}z  d \sigma d t .$ 
	Integrating by parts and using the fact that  
	\begin{equation}\label{Prxi0.0}
	\left|\partial_{t}\left(e^{-2 s \alpha} \xi^{-2}\right)\right| \leq C s^{2} e^{-2 s \alpha},
	\end{equation}
	we get
	\begin{align*}
	s^{-2} \lambda^{-2} \int_{\Omega_{T}} e^{2 s \alpha} \xi^{-2} z \partial_{t} z d x d t  \leq C \int_{\Omega_{T}} e^{2 s \alpha}|z|^{2} d x d t.
	\end{align*}
	The other terms can be estimated, by Young inequality, as follows. Using the fact that $\nabla \eta_{0}$ is a bounded function on $\overline{\Omega}$, we have
	{\small
		\begin{align*}
		&2 s^{-1} \lambda^{-1} \int_{\Omega_{T}} e^{2 s \alpha} \xi^{-1} \nabla \eta_{0} \cdot \mathcal{A}\nabla z z d x d t+2 s^{-2} \lambda^{-1} \int_{\Omega_{T}} e^{2 s \alpha} \xi^{-2} \nabla \eta_{0} \cdot \mathcal{A}\nabla z z d x d t \\
		&\leq C\left(\int_{\Omega_{T}} e^{2 s \alpha}|z|^{2} d x d t+s^{-2} \lambda^{-2} \int_{\Omega_{T}} e^{2 s \alpha} \xi^{-2}|\mathcal{A}\nabla z|^{2} d x d t\right) \\
		&\quad+C\left(\epsilon^{-1} \int_{\Omega_{T}} e^{2 s \alpha}|z|^{2} d x d t  + \epsilon s^{-4} \lambda^{-2}  \int_{\Omega_{T}} e^{2 s \alpha} \xi^{-4}|\mathcal{A}\nabla z|^{2} d x d t\right) \\
		&\leq  C\epsilon^{-1} \int_{\Omega_{T}} e^{2 s \alpha}|z|^{2} d x d t+ C\epsilon \, \lambda^{-2} \int_{\Omega_{T}} e^{2 s \alpha} \xi^{-2}|\mathcal{A}\nabla z|^{2} d x d t
		\end{align*}}
	for  $s\geq s_{0}$ and $\lambda \geq \lambda_{0}$, and $\epsilon$ small enough.
	Using again Young inequality, we find
	\begin{align*}
	s \lambda^{2} \int_{\Omega_{T}} \xi \varphi z d x d t &\leq \frac{1}{2}\left(s^{-1} \int_{\Omega_{T}} e^{2 s \alpha} \xi^{-1}|z|^{2} d x d t+s^{3} \lambda^{4} \int_{\Omega_{T}} e^{-2 s \alpha} \xi^{3}|\varphi|^{2} d x d t\right)\\
	&\leq \frac{1}{2}\left(\int_{\Omega_{T}} e^{2 s \alpha}|z|^{2} d x d t+s^{3} \lambda^{4} \int_{\Omega_{T}} e^{-2 s \alpha} \xi^{3}|\varphi|^{2} d x d t\right)
	\end{align*}
	and 
	{\small
		\begin{align*}
		s^{-2} \lambda^{-2} \int_{\omega_{T}} e^{2 s \alpha} \xi^{-2} v z d x d t &\leq \frac{1}{2}\left(\int_{\Omega_{T}} e^{2 s \alpha}|z|^{2} d x d t+s^{-4} \lambda^{-4} \int_{\omega_{T}} e^{2 s \alpha} \xi^{-4}|v|^{2} d x d t\right)\\
		&\leq \frac{1}{2}\left(\int_{\Omega_{T}} e^{2 s \alpha}|z|^{2} d x d t+s^{-3} \lambda^{-4} \int_{\omega_{T}} e^{2 s \alpha} \xi^{-3}|v|^{2} d x d t\right).
		\end{align*}}
	Now, \eqref{Kh2.33} yields that
	{\small
		\begin{align*}
		& s^{-2} \lambda^{-2} \int_{\Omega_{T}} e^{2 s \alpha} \xi^{-2}\nabla z\cdot\mathcal{A}\nabla z d x\,d t+\mathbf{B}_{1}\\
		&\leq C\epsilon^{-1} \int_{\Omega_{T}} e^{2 s \alpha}|z|^{2} d x d t+\frac{1}{2} s^{-3} \lambda^{-4} \int_{\omega_{T}} e^{2 s \alpha} \xi^{-3}|v|^{2} d x\,d t\\
		&+\frac{1}{2} s^{3} \lambda^{4} \int_{\Omega_{T}} e^{-2 s \alpha} \xi^{3}|\varphi|^{2} d x\,d t+\epsilon  \lambda^{-2} \int_{\Omega_{T}} e^{2 s \alpha} \xi^{-2}|\mathcal{A}\nabla z|^{2} d x\,dt. 
		\end{align*}}
	Choosing $\epsilon$ small enough and using the fact that $\mathcal{A}$ is bounded and uniformly elliptic, we obtain
	\begin{align}
	& s^{-2} \lambda^{-2} \int_{\Omega_{T}} e^{2 s \alpha} \xi^{-2}|\nabla z|^{2} d x\,d t+\mathbf{B}_{1}\nonumber\\
	&\leq C\left(s^{3} \lambda^{4} \int_{\Omega_{T}} e^{-2 s \alpha} \xi^{3}|\varphi|^{2} d x\,d t+s^{3} \lambda^{3} \int_{\Gamma_{T}} e^{-2 s \alpha} \xi^{3}\left|\varphi_{\Gamma}\right|^{2} d \sigma\,d t\right) \label{Kh2.036}.
	\end{align}
	{\bf{\text { Step } 3: \text { Estimate of the first order boundary term.}}} To estimate the first 
	order boundary term, we multiply by $s^{-2} \lambda^{-2} e^{2 s \alpha} \xi^{-2} z_{\Gamma}$ the second equation of  \eqref{Kh.28} verified by $z_{\Gamma}$ on $\Gamma_{T}$ and integrate by parts, and since $\nabla_{\Gamma} \alpha=\nabla_{\Gamma} \xi=0,$ we obtain
	{\small
		\begin{align*}
		&s^{-2} \lambda^{-2} \int_{\Gamma_{T}} e^{2 s \alpha} \xi^{-2} z_{\Gamma} \partial_{t} z_{\Gamma} d \sigma\,d t+ s^{-2} \lambda^{-2} \int_{\Gamma_{T}} e^{2 s \alpha} \xi^{-2}\nabla_{\Gamma} z\cdot\mathcal{A}_{\Gamma}\nabla_{\Gamma} z_{\Gamma} d \sigma\,d t+\mathbf{B}_{2} \\
		&=s \lambda \int_{\Gamma_{T}} \xi \varphi_{\Gamma} z_{\Gamma} d \sigma\,d t,
		\end{align*}}
	where $\displaystyle \mathbf{B}_{2}:=s^{-2} \lambda^{-2} \int_{\Gamma_{T}} e^{2 s \alpha} \xi^{-2} z_{\Gamma} \partial^{\mathcal{A}}_{\nu} z d \sigma \,dt=-\mathbf{B}_{1} .$ 	Integration by parts and \eqref{Prxi0.0} yield
	{\small
		\begin{align*}
		&-s^{-2} \lambda^{-2} \int_{\Gamma_{T}} e^{2 s \alpha} \xi^{-2} z_{\Gamma} \partial_{t} z_{\Gamma} d \sigma d t=\frac{1}{2} s^{-2} \lambda^{-2} \int_{\Gamma_{T}} \partial_{t}\left(e^{2 s \alpha} \xi^{-2}\right)\left|z_{\Gamma}\right|^{2} d \sigma dt\\
		&\leq C \int_{\Gamma_{T}} e^{2 s \alpha}\left|z_{\Gamma}\right|^{2} d \sigma d t\\
		&\leq C\left(s^{3} \lambda^{4} \int_{\Omega_{T}} e^{-2 s \alpha} \xi^{3}|\varphi|^{2} d x d t+s^{3} \lambda^{3} \int_{\Gamma_{T}} e^{-2 s \alpha} \xi^{3}\left|\varphi_{\Gamma}\right|^{2} d \sigma d t\right).
		\end{align*}}
	On the other hand, by Young inequality and \eqref{Kh2.032},  we have
	\begin{align*}
	&\text { s } \int_{\Gamma_{T}} \xi \varphi_{\Gamma} z_{\Gamma} d \sigma d t \leq \frac{1}{2}\left(\int_{\Gamma_{T}} e^{2 s \alpha}\left|z_{\Gamma}\right|^{2} d \sigma d t+s^{2} \lambda^{2} \int_{\Gamma_{T}} e^{-2 s \alpha} \xi^{2}\left|\varphi_{\Gamma}\right|^{2} d \sigma d t\right) \\
	&\leq \frac{1}{2}\left(\int_{\Gamma_{T}} e^{2 s \alpha}\left|z_{\Gamma}\right|^{2} d \sigma d t+s^{3} \lambda^{3} \int_{\Gamma_{T}} e^{-2 s \alpha} \xi^{3}\left|\varphi_{\Gamma}\right|^{2} d \sigma d t\right) \\
	&\leq C\left( s^{3} \lambda^{4} \int_{\Omega_{T}} e^{-2 s \alpha} \xi^{3}|\varphi|^{2} d x d t+ s^{3} \lambda^{3} \int_{\Gamma_{T}} e^{-2 s \alpha} \xi^{3}\left|\varphi_{\Gamma}\right|^{2} d \sigma d t\right).
	\end{align*}
	
	Hence, since $\mathcal{A}_{\Gamma}$ is uniformly  elliptic,  for $\tilde{s}$, $\tilde{\lambda}$ large if needed, we find
	\begin{align}
	& s^{-2} \lambda^{-2} \int_{\Gamma_{T}} e^{2 s \alpha} \xi^{-2}\left|\nabla_{\Gamma} z\right|^{2} d \sigma d t+\mathbf{B}_{2}\nonumber\\
	&=s \lambda \int_{\Gamma_{T}} \xi \varphi_{\Gamma} z_{\Gamma} d \sigma d t-s^{-2} \lambda^{-2} \int_{\Gamma_{T}} e^{2 s \alpha} \xi^{-2} z \partial_{t} z d \sigma d t\nonumber\\
	&\leq C\left(s^{3} \lambda^{4} \int_{\Omega_{T}} e^{-2 s \alpha} \xi^{3}|\varphi|^{2} d x d t+s^{3} \lambda^{3} \int_{\Gamma_{T}} e^{-2 s \alpha} \xi^{3}\left|\varphi_{\Gamma}\right|^{2} d \sigma d t\right).\label{Kh2.36}
	\end{align}
	Now, summing \eqref{Kh2.032}, \eqref{Kh2.032} and \eqref{Kh2.36}, and using the fact that $\mathbf{B}_{1}+\mathbf{B}_{2}=0$,  we obtain the desired estimate.  \qed
\end{proof}

\begin{proof}[Proof of Theorem \ref{ThKh2.7}]
	The proof will be done in three
	steps.\\
	{\bf{ Step 1: Estimate of the two first terms.}}
	Let  $Z=\left(z, z_{\Gamma}\right)$ the weak solution of \eqref{Kh.28}. Following \cite{KhMa19},   we can easily  prove that
	{\small
		\begin{align}
		&s^{3} \lambda^{4} \int_{\Omega_{T}} e^{-2 s \alpha} \xi^{3}|\varphi|^{2} d x d t+s^{3} \lambda^{3} \int_{\Gamma_{T}} e^{-2 s \alpha} \xi^{3}\left|\varphi_{\Gamma}\right|^{2} d \sigma d t\nonumber\\
		&\leq C\left(s^{3} \lambda^{4} \int_{\omega_{T}} e^{-2 s \alpha} \xi^{3}|\varphi|^{2} d x d t+\int_{\Omega_{T}} e^{-2 s \alpha}\left|F_{0}\right|^{2} d x d t+s^{2} \lambda^{2} \int_{\Omega_{T}} e^{-2 s \alpha} \xi^{2}\|F\|_{\mathbb{R}^{N}}^{2} d x d t\right.\nonumber\\
		&\left.+\int_{\Gamma_{T}} e^{-2 s \alpha}\left|F_{0, \Gamma}\right|^{2} d \sigma d t+s^{2} \lambda^{2} \int_{\Gamma_{T}} e^{-2 s \alpha} \xi^{2}\left\|F_{\Gamma}\right\|_{\mathbb{R}^{N}}^{2} d \sigma d t\right).\label{Kh1.40}
		\end{align}}
	{\bf{Step 2. Estimate of the first order terms.}} We use the fact that	
	$\Phi$ is a weak solution to \eqref{1.5} and that $s \lambda^{2} e^{-2 s \alpha} \xi \Phi=\left(s \lambda^{2} e^{-2 s \alpha} \xi \varphi, s \lambda^{2} e^{-2 s \alpha} \xi \varphi_{\Gamma}\right) \in$
	$L^{2}\left(0, T ; \mathbb{H}^{1}\right)$ to obtain from \eqref{Kh1.12} with $V=s \lambda^{2} e^{-2 s \alpha} \xi \Phi,$ that
	{\small
		\begin{align*}
		& s \lambda^{2} \int_{\Omega_{T}} \mathcal{A}\nabla \varphi\cdot \nabla\left(e^{-2 s \alpha} \xi\varphi\right) d x d t+ s \lambda^{2} \int_{\Gamma_{T}} \mathcal{A}_{\Gamma}\nabla_{\Gamma} \varphi_{\Gamma}\cdot \nabla_{\Gamma}\left(e^{-2 s \alpha} \xi \varphi_{\Gamma}\right) d \sigma d t\\
		&+\int_{0}^{T}\left\langle\partial_{t}\left(e^{-2 s \alpha} \xi \varphi\right), \varphi\right\rangle_{\left(H^{1}(\Omega)\right)^{\prime}, H^{1}(\Omega)} d t\\	
		&+ s \lambda^{2} \int_{0}^{T}\left\langle\partial_{t}\left(e^{-2 s \alpha} \xi \varphi_{\Gamma}\right), \varphi_{\Gamma}\right\rangle_{H^{-1}(\Gamma), H^{1}(\Gamma)} dt\\
		&=-\int_{0}^{T}\langle F_{0}-\operatorname{div} F, z(t)\rangle_{\left(H^{1}(\Omega)\right)^{\prime}, H^{1}(\Omega)} d t \\	
		&-\int_{0}^{T}\langle F \cdot\nu+F_{0, \Gamma}-\operatorname{div}_{\Gamma} F_{\Gamma}, z_{\Gamma}\rangle _{H^{-1}(\Gamma), H^{1}(\Gamma)} d t.
		\end{align*}}
	Integrating by parts in time and space, we obtain
	{\small
		\begin{align*}
		&s \lambda^{2} \int_{\Omega_{T}} e^{-2 s \alpha} \xi \mathcal{A}\nabla \varphi \cdot \nabla \varphi d x d t+\frac{s \lambda^{2}}{2} \int_{\Omega_{T}}\left(\partial_{t}\left(e^{-2 s \alpha} \xi\right)-\mathrm{div}(\mathcal{A}\nabla e^{-2 s \alpha} \xi)\right)|\varphi|^{2} d x d t+\mathbf{B}_{3}\\
		&+s \frac{\lambda^{2}}{2} \int_{\Gamma_{T}}\left(\partial_{t}\left(e^{-2 s \alpha} \xi\right)-\mathrm{div}_{\Gamma}(\mathcal{A}_{\Gamma}\nabla_{\Gamma} (e^{-2 s \alpha} \xi))\right)\left|\varphi_{\Gamma}\right|^{2} d \sigma d t\\
		&   +  s \lambda^{2} \int_{\Gamma_{T}} e^{-2 s \alpha} \xi\mathcal{A}\nabla_{\Gamma}\cdot \nabla_{\Gamma} \varphi_{\Gamma}\, d \sigma d t\\
		&=-s \lambda^{2} \int_{\Omega_{T}} e^{-2 s \alpha}\left(\xi F_{0} \varphi-\xi F \cdot \nabla \varphi\right) d x d t+s \lambda^{2} \int_{\Omega_{T}} F \cdot \nabla\left(e^{-2 s \alpha} \xi\right) \varphi d x d t\\ 
		&-s \lambda^{2} \int_{\Gamma_{T}} e^{-2 s \alpha}\left(\xi F_{0, \Gamma} \varphi-\xi F_{\Gamma} \cdot \nabla_{\Gamma} \varphi_{\Gamma}\right) d \sigma d t+s \lambda^{2} \int_{\Gamma_{T}} F_{\Gamma} \cdot \nabla_{\Gamma}\left(e^{-2 s \alpha} \xi\right) \varphi_{\Gamma} d \sigma d t,
		\end{align*}}
	where,  we have set $
	\mathbf{B}_{3}:=-\frac{ s \lambda^{2}}{2} \int_{\Gamma_{T}} \partial^{\mathcal{A}}_{\nu}\left(e^{-2 s \alpha} \xi\right) \varphi_{\Gamma}^{2} d \sigma d t.$
	The same calculations as  in \cite{KhMa19}  leads to
	{\small	{\small
			\begin{align*}
			&s \lambda^{2} \int_{\Omega_{T}} e^{-2 s \alpha} \xi\mathcal{A}\nabla \varphi\cdot\nabla \varphi d x d t+s \lambda^{2} \int_{\Gamma_{T}} e^{-2 s \alpha} \xi\mathcal{A}_{\Gamma}\nabla_{\Gamma} \varphi_{\Gamma}\cdot\nabla_{\Gamma} \varphi_{\Gamma} d \sigma d t\\
			&\leq C\Big(s^{2} \lambda^{2} \int_{\Omega_{T}} \xi^{2}\|F\|_{\mathbb{R}^{N}}^{2} d x d t+s^{3} \lambda^{4} \int_{\Omega_{T}} e^{-2 s \alpha} \xi^{3}|\varphi|^{2} d x d t\\
			& + s^{-1} \int_{\Omega_{T}} e^{-2 s \alpha} \xi^{-1}\left|F_{0}\right|^{2} d x d t\Big) + C\Big(s \lambda^{2} \int_{\Gamma_{T}} \xi\left\|F_{\Gamma}\right\|_{\mathbb{R}^{N}}^{2} d \sigma d t\\
			&+s^{3} \lambda^{4} \int_{\Gamma_{T}} e^{-2 s \alpha} \xi^{3}\left|\varphi_{\Gamma}\right|^{2} d \sigma d t+s^{-1} \int_{\Gamma_{T}} e^{-2 s \alpha} \xi^{-1}\left|F_{0, \Gamma}\right|^{2} d \sigma\Big).
			\end{align*}}}
	We can then conclude  by the fact that   $\mathcal{A}$ and $\mathcal{A}_{\Gamma}$ are uniformly elliptic. This achieves the proof. \qed
\end{proof}
Now, we are in position to deal with the  cost of  approximate controllability issue.
\section{ Cost of interior approximate controllability}\label{subs.1}
As mentioned in the introduction,   the  system (\ref{1.1}) is approximately controllable. In other words, for given $Y_0\in \mathbb{L}^2$, a final state $Y_1\in \mathbb{L}^2$ and $\varepsilon>0$, there exists a control $v\in L^2(\omega_T)$ such that the solution of (\ref{1.1}) satisfies 
\begin{equation}\label{5.1}
\parallel Y(T)- Y_1\parallel_{\mathbb{L}^2}\leq \varepsilon.
\end{equation}
In particular, the set of admissible controls \,$\mathcal{U}_{ad}(Y_0, Y_1, \varepsilon, \omega)$ 
\begin{equation*}\label{0.3}
\mathcal{U}_{ad}(Y_0, Y_1, \varepsilon,\omega)=\big\{ v\in L^2(\omega_T): \text{the\, solution} \,Y \, \text{to} \,\,\eqref{1.1}\,\, \text{satisfies} \,\, \eqref{5.1}\big\}.
\end{equation*}
is nonempty. 
Let us introduce the following quantity, which measures the cost of approximate controllability or, more precisely, the cost of achieving \eqref{5.1}
\begin{equation}\label{5.3}
\mathcal{C}(Y_0, Y_1, \varepsilon, \omega)= \inf_{v\in \mathcal{U}_{ad}(Y_0, Y_1, \varepsilon, \omega)} \lVert v\lVert_{L^2(\omega_T)}.
\end{equation}
The first main aim of this paper is to obtain explicit bounds of  $\mathcal{C}(Y_0, Y_1, \varepsilon, \omega)$.

Taking into account that system \eqref{1.1} is linear, one can assume, without loss
of generality, that $Y_0=0$,  since
\begin{equation*}\label{0.6}
\mathcal {U}_{ad}(Y_0, Y_1, \varepsilon, \omega)= \mathcal {U}_{ad}(0, Z_1, \varepsilon, \omega)
\end{equation*}
for  $Z_1=Y_1-Z(T)$, with $Z$ is the solution of (\ref{1.1}) with $v=0$. 
We can prove, by means of Fenchel-Rockaffelar's duality Theorem \ref{RF.1}, that 
\begin{equation*}
\frac{1}{2}\;\mathcal{C}(0, Y_1, \varepsilon, \omega)^2=\inf_{v\in \mathcal {U}_{ad}(0, Y_1, \varepsilon, \omega)} \frac{1}{2} \lVert v\lVert^2_{L^2(\omega_T)}= - \inf_{\Phi_{T}\in \mathbb{L}^{2}} J^{\varepsilon}_{Y_{1},\omega}(\Phi_T),
\end{equation*}
where
\begin{equation*}\label{5.7}
J^{\varepsilon}_{Y_1,\omega}(\Phi_T)=  \frac{1}{2}\int_{\omega_{T}}|\varphi|^2\, dx dt + \varepsilon\|\Phi_T\|_{\mathbb{L}^2} -\big\langle Y_1,\Phi_T\big\rangle_{\mathbb{L}^2},
\end{equation*}
and $\Phi=(\varphi,\varphi_{\Gamma})$ is  the solution of the adjoint system (\ref{1.5}) with final data $\Phi_T$,   see \cite{GlLi08} for a detailed description to the above technique.
\begin{remark}
	Point out that the admissible set $\mathcal {U}_{ad}(Y_0, Y_1, \varepsilon, \omega)$ is convex and closed, and then the problem \eqref{5.3} admits a unique solution,  but it can be a difficult task to solve and characterize  directly its solution, due to  the nature of  constraints that are involved. This is why  the Fenchel-Rockaffelar duality	idea is often used.
\end{remark}

In the sequel we prove that $J^{\varepsilon}_{Y_{1},\omega}$ achieves its minimum at an element $\widehat{\Phi}_T$,  and by means of the solution of the adjoint system  associated to $\widehat{\Phi}_T$ we deduce the norm optimal control.
\begin{lemma}\label{ll4.1}
	The map $J^{\varepsilon}_{Y_1,\omega}$  is continuous, strictly convex and coercive in $\mathbb{L}^2$, and then there is $\widehat{\Phi}_{T}=(\widehat{\varphi}_T,\widehat{\varphi}_{\Gamma,T})\in\mathbb{L}^{2}$ such that
	\begin{equation}\label{5.8}
	J^{\varepsilon}_{Y_1,\omega}(\widehat{\Phi}_T)= \inf_{\Phi_{T}\in \mathbb{L}^{2}} J^{\varepsilon}_{Y_1,\omega}(\Phi_T).
	\end{equation}
\end{lemma}
\begin{proof} 
	It is easy to see that $J_{\varepsilon,Y_{1}}$ is strictly convex and, by  the inequality \eqref{est2.10}, $J_{\varepsilon,Y_{1}}$ is continuous in $\mathbb{L}^2$. Hence, by {\cite[Corollary \MakeUppercase{\romannumeral 3}.20, p.58]{C1}}, the existence of a minimum is ensured if $J_{\varepsilon,Y_{1}}$  is coercive,  i.e.,
	\begin{equation}\label{5.88}
	\lim_{\|\Phi_T\|\rightarrow +\infty}J^{\varepsilon}_{Y_1,\omega}(\Phi_T)= +\infty. 
	\end{equation}
	It suffices to prove that
	\begin{equation}\label{5.9}
	\liminf_{\|\Phi_T\|\rightarrow+\infty}\frac{J^{\varepsilon}_{Y_1,\omega}(\Phi_T)}{\|\Phi_T\|}\geq \varepsilon, 
	\end{equation}
	since, obviously (\ref{5.9}) implies  (\ref{5.88}).
	To  prove \eqref{5.9},  let $(\Phi^n_T)= (\varphi_T^n, \varphi^n_{T,\Gamma})\subset\mathbb{L}^2$ be a sequence of initial data for the adjoint system with $\|\Phi^n_T\|_{\mathbb{L}^2}\longrightarrow +\infty$, and 
	set $ \widetilde{\Phi}^n_T = \frac{\Phi^n_T}{\|\Phi^n_T\|_{\mathbb{L}^2}}.$
	On the other hand, let $\widetilde{\Phi}^n=(\widetilde{\varphi}^n, \widetilde{\varphi}_{\Gamma}^n)$ the solution of the adjoint problem with final data $ \widetilde{\Phi}^n_T$. Then,
	$$ \frac{J^{\varepsilon}_{Y_1,\omega}(\Phi^n_T)}{\|\Phi^n_T\|_{\mathbb{L}^2}} = \frac{1}{2}\|\Phi^n_T\|_{\mathbb{L}^2} \int_{\omega_T}|\widetilde{\varphi}^n|^2\,dx\,dt  + \varepsilon \, - \, \big\langle Y_1,\widetilde{\Phi}^n_T\big\rangle_{\mathbb{L}^2}.$$
	We  discuss the  following two cases :
	If  $\displaystyle\liminf_{n\longrightarrow +\infty}\int_{\omega_T}|\widetilde{\varphi}^n|^2\,dx\,dt>0$, we obtain 
	$$\frac{J^{\varepsilon}_{Y_1,\omega}(\Phi^n_T)}{\|\Phi^n_T\|_{\mathbb{L}^2}} \longrightarrow +\infty, \quad \text{as}\,\, n\longrightarrow +\infty .$$ 
	If  now $\displaystyle\liminf_{n\longrightarrow +\infty}\int_{\omega_T}|\widetilde{\varphi}^n|^2\,dx\,dt=0,$ since the sequence $(\widetilde{\Phi}^n_T)$ is bounded in $\mathbb{L}^2,$ by extracting a subsequence (still
	denoted by the index $n$), we deduce that $(\widetilde{\varphi}^n_T,\widetilde{\varphi}^n_{\Gamma,T})=\widetilde{\Phi}^n_T\rightharpoonup \Psi_T:=(\psi_T,\psi_{\Gamma,T})$ weakly in $\mathbb{L}^2$ and, by using the linearity and the continuity of solution operator, we get $(\widetilde{\varphi}^n,\widetilde{\varphi}^n_{\Gamma})=\widetilde{\Phi}^n\rightharpoonup \Psi=(\psi,\psi_{\Gamma})$ weakly in $L^2( 0,T;\mathbb{L}^2)$,  where $\Psi$ is the solution of the adjoint system \eqref{1.5} with final data $\Psi_T.$ In particular,  $\widetilde{\varphi}^n\rightharpoonup \psi$ weakly in $L^2( \omega_T), $ and 
	$$\int_{\omega_T}\psi^2\,dx dt \leq \liminf_{n\longrightarrow +\infty}\int_{\omega_T}|\widetilde{\varphi}^n|^2\,dx\,dt=0,$$
	and therefore  $\psi = 0$ on $\omega_T.$ Then, by the unique continuation property \eqref{01.3}, we get  $\psi = 0$ on $\Omega_T$ and $\psi_{\Gamma} = 0$ on $\Gamma_T.$ We can then deduce that $\Psi_T =0$, and, in particular,  $\big\langle Y_1,\widetilde{\Phi}^n_T\big\rangle_{\mathbb{L}^2}\longrightarrow 0.$
	Hence
	$$   \liminf_{n\rightarrow\infty}  \frac{J^{\varepsilon}_{Y_1,\omega}(\Phi^n_T)}{\|\Phi^n_T\|_{\mathbb{L}^2}}\geq \liminf_{n\rightarrow\infty}\big(\varepsilon- \big\langle Y_1,\widetilde{\Phi}^n_T\big\rangle_{\mathbb{L}^2}\big)=\varepsilon.$$
	This concludes the proof. \qed
\end{proof}

The result  \eqref{5.8} gives rise to the solution of the problem (\ref{5.3}), more precisely we have the following results.
\begin{theorem}\label{lem2.5}
	(i)\; 	Let  $\widehat{\Phi}_T$ be the minimum of $J^{\varepsilon}_{Y_{1}, \omega}$ in $\mathbb{L}^2$, and  $\widehat{\Phi}=(   \widehat{\varphi}, \widehat{\varphi}_\Gamma )$ the solution of \eqref{1.5} with final data $\widehat{\Phi}_T$.   Then,  $\widehat{v}=1_\omega\widehat{ \varphi}$ is a control of \eqref{1.1},  i.e.,
	\begin{equation*}
	\|Y(T,\widehat{v} )-Y_1\|_{\mathbb{L}^2}\leq \varepsilon.
	\end{equation*}
	(ii)\; The control $\widehat{v}$  is an optimal one, that is to say
	$$\|\widehat{\varphi}\|_{L^2(\omega_T)} =\inf_{v\in \mathcal{U}_{\text{ad}}(0,Y_1,\varepsilon,\omega)}\|v\|_{L^2(\omega_T)}.$$ 
\end{theorem}
\begin{proof}
	(i)\; 	Suppose that $J^{\varepsilon}_{Y_{1}, \omega}$ attains its minimum value at $\widehat{\Phi}_T\in \mathbb{L}^2$. Then,  for any $\Psi_T\in \mathbb{L}^2$ and $h\in \mathbb{R}$, we have
	$$  J^{\varepsilon}_{Y_1,\omega}(\widehat{\Phi}_T ) \leq  J^{\varepsilon}_{Y_1,\omega}(\widehat{\Phi}_T + h\Psi_T).$$   
	By definition of  $J^{\varepsilon}_{Y_1,\omega}$ and  the linearity of solution operator, we have 
	{\small
		\begin{align*} 
		J^{\varepsilon}_{Y_1,\omega}(\widehat{\Phi}_T + h\Psi_T)& = \frac{1}{2}\|\widehat{\varphi} + h {\psi}\|^2 _{L^2(\omega_T)}  + \varepsilon\|  \widehat{\Phi}_T + h {\Psi_T}\|_{\mathbb{L}^2}  \, -\, \big\langle Y_1,\widehat{\Phi}_T +h \Psi_T \big\rangle_{\mathbb{L}^2} \\
		&= \frac{1}{2}\| \widehat{\varphi}\|^2_{L^2(\omega_T)}  + \frac{h^2}{2}  \| \psi\|^2_{L^2(\omega_T)}  + h\langle \widehat{\varphi}, \psi \rangle_{L^2(\omega_T)}    + \varepsilon\|  \widehat{\Phi}_T + h {\Psi}\|_{\mathbb{L}^2} \\
		&\quad - \big\langle Y_1,\widehat{\Phi}_T +h \Psi_T \big\rangle_{\mathbb{L}^2}.
		\end{align*}}
	Here $\Psi=(\psi, \psi_\Gamma)$ is, as usual,  the solution of adjoint system \eqref{1.5} with final data $\Psi_T.$
	Thus              
	\begin{equation*}
	0\leq \big[\varepsilon\|  \widehat{\Phi}_T + h {\Psi_T}\|_{\mathbb{L}^2}- \|  \widehat{\Phi}_T \|_{\mathbb{L}^2}\big] + \frac{h^2}{2}  \| \psi\|^2_{L^2(\omega_T)} + h\big(\langle \widehat{\varphi}, \psi \rangle_{L^2(\omega_T)} \,-\,\big\langle Y_1,\Psi_T\big\rangle_{\mathbb{L}^2}\big).
	\end{equation*}
	Since $ \|  \widehat{\Phi}_T + h {\Psi_T}\|_{\mathbb{L}^2}-  \|  \widehat{\Phi}_T \|_{\mathbb{L}^2} \leq  |h|\|  \Psi_T \|_{\mathbb{L}^2}$, we obtain, 	for all  $h\in\mathbb{R}$ and  $\Psi_T\in \mathbb{L}^2$, 
	\begin{equation*}
	0\leq \varepsilon |h|\| {\Psi_T}\|_{\mathbb{L}^2}  + \frac{h^2}{2}  \| \psi\|^2_{L^2(\omega_T)} + h\big(\langle \widehat{\varphi}, \psi \rangle_{L^2(\omega_T)} \, -\,\big\langle Y_1,\Psi_T\big\rangle_{\mathbb{L}^2}\big).
	\end{equation*}
	Dividing by $h > 0$ and  passing to the limit $h \longrightarrow 0$,  we obtain
	\begin{equation*}
	0\leq \varepsilon \| {\Psi_T}\|_{\mathbb{L}^2}   + \big(\langle \widehat{\varphi}, \psi \rangle_{L^2(\omega_T)} \, -\,\big\langle Y_1,\Psi_T\big\rangle_{\mathbb{L}^2}\big).
	\end{equation*} 
	The same calculations with $h < 0$ gives,  for all $\Psi_T\in \mathbb{L}^2$,        
	\begin{equation}\label{5.13}
	\big|\langle \widehat{\varphi}, \psi \rangle_{L^2(\omega_T)} \, -\,\big\langle Y_1,\Psi_T\big\rangle_{\mathbb{L}^2}\big|\leq \varepsilon \| {\Psi_T}\|_{\mathbb{L}^2}.
	\end{equation}
	Taking the control $v= \widehat{\varphi}$ in (\ref{1.1}), and multiplying the adjoint system (\ref{1.5}) by $\Psi$ and integrating by parts,  we deduce 
	$ \langle \widehat{\varphi}, \psi \rangle_{L^2(\omega_T)}= \big\langle Y(T),\Psi_T\big\rangle_{\mathbb{L}^2} .$
	This and  (\ref{5.13}) yield our aim
	$$   \|Y(T)-Y_1\|_{\mathbb{L}^2}\leq \varepsilon.$$
	For (ii), observe, thanks to Fenchel-Rockafellar duality Theorem \ref{RF.1} and the definition of $J^{\varepsilon}_{Y_1,\omega}$, that the minimization problem 
	\begin{equation*}      
	\|\widehat{\varphi}\| =\inf_{v\in \mathcal{U}_{\text{ad}}(0,Y_1,\varepsilon, \omega)}\|v\|_{L^2(\omega_T)} 
	\end{equation*} 
	is equivalent to the following equality             
	\begin{equation*}      
	\|\widehat{\varphi}\|^2_{L^2(\omega_T)} +\varepsilon\|\widehat{\Phi}_T\|_{\mathbb{L}^2} - \big\langle Y_1,\widehat{\Phi}_T\big\rangle_{\mathbb{L}^2}=0. 
	\end{equation*}
	Since $\widehat{v}:=\widehat{\varphi}\in \mathcal{U}_{\text{ad}}(0,Y_1,\varepsilon,\omega),$  we deduce easily that  
	$$\|\widehat{\varphi}\|^2_{L^2(\omega_T)} +\varepsilon\|\widehat{\Phi}_T\|_{\mathbb{L}^2} - \big\langle Y_1,\widehat{\Phi}_T\big\rangle_{\mathbb{L}^2}\geq 0.$$
	To finish the proof, it remains to show the other inequality, to do this end, we make use of the following lemma, see {\cite{B2}} for the proof.
	\begin{lemma}[\cite{B2}]\label{l5.4}
		Let $H$ a Hilbert space and $K\subset H$ a convex closed. Let $J_1, J_2: K\longrightarrow \mathbb{R}$ be two continuous convex functions. We assume that $J_1$ is differentiable on int(K). Then  the following two conditions are equivalent
		\begin{itemize}
			\item[(a)]  $u\in K,\quad J_1(u) +J_2(u)= \inf\limits_{ v\in K} \big[J_1(v) +J_2(v)\big]$,
			\item[(b)]  $u\in K,\quad J^{'}_1(u)\cdot(v-u) + J_2(v) - J_2(u) \geq 0,  \quad \forall v \in K.$
		\end{itemize}          
	\end{lemma} 
	\vspace{0.3cm}
	\par Let us apply  Lemma \ref{l5.4} for
	$$J_1(\Phi_T)= \frac{1}{2}\|\varphi\|^2_{L^2(\omega_T)} - \big\langle Y_1,\Phi_T \big\rangle_{\mathbb{L}^2}\quad \text{and}\quad
	J_2(\Phi_T)= \varepsilon\|\Phi_T\|_{\mathbb{L}^2},\quad\Phi_T\in {\mathbb{L}^2} .$$ 
	We have \,$ J^{\varepsilon}_{Y_1,\omega}= J_1+J_2$, and 
	\begin{equation} \label{5.15}
	\small J^{\varepsilon}_{Y_1, \omega}(\widehat{\Phi}_T) =\min_{\Phi_T\in \mathbb{L}^2(\Omega)} J^{\varepsilon}_{Y_1, \omega}(\Phi_T)  \Longleftrightarrow J^{'}_1(\widehat{\Phi}_T)\cdot (v-\widehat{\Phi}_T) + J_2(v) - J_2(\widehat{\Phi}_T) \geq 0,\, \,  \forall\, v \in \mathbb{L}^2.                                 
	\end{equation}
	It is  easy to see that
	\begin{equation} \label{5.16}
	J_1(\Phi_T)=\frac{1}{2}\|\Pi_1\circ S(\Phi_T)\|^2_{L^2(\omega_T)} -\big\langle Y_1, \Phi_T\big\rangle_{\mathbb{L}^2},
	\end{equation}
	where $S$ is the solution operator  of the adjoint system \eqref{1.5} i.e., $S(\Phi_T)=\Phi$ and $\Pi_1$ is the first projection from $L^2(\omega_T)\times L^2(\Gamma_T)$ into $L^2(\omega_T)$, that is, 
	$\Pi_1(u,v)=u$ for all $(u,v)\in L^2(\omega_T)\times L^2(\Gamma_T)$.
	Hence,  from (\ref{5.16}),  for all $K_0\in \mathbb{L}^2$, one has
	{\small
		\begin{equation}\label{5.17}
		J^{'}_1(\widehat{\Phi}_T)\cdot(K_0)  = \big\langle \Pi_1(S(\Phi_T)) , \Pi_1(S(K_0))\big\rangle_{\mathbb{L}^2} - \,\, \big\langle Y_1,K_0 \big\rangle_{\mathbb{L}^2} = \langle \widehat{\varphi},k \rangle_{L^2(\omega_T)}-\,\,\big\langle Y_1,K_0 \big\rangle_{\mathbb{L}^2},
		\end{equation}}
	with $\widehat{\Phi}=(\widehat{\varphi},\widehat{\varphi}_\Gamma)$ is  (respectively $K=(k,k_\Gamma)$) the solution of the adjoint system (\ref{1.5}) with final  data $\widehat{\Phi}_T$ (respectively $K_0$).
	Taking $v=0$ in the right hand side of (\ref{5.15}),  it follows
	\begin{equation*}
	J^{'}_1(\widehat{\Phi}_T)\cdot(-\widehat{\Phi}_T)  +J_2(\widehat{\Phi}_T) \geq 0,                      
	\end{equation*}
	and using (\ref{5.17}), we have
	\begin{equation*}
	-\|\widehat{\varphi}\|^2_{L^2(\omega_T)}  + \big\langle Y_1,\widehat{\Phi}_T \big\rangle_{\mathbb{L}^2} -  \varepsilon\|\widehat{\Phi}_T\|_{\mathbb{L}^2} \geq 0,
	\end{equation*}
	which finishes the proof.       \qed      
\end{proof}
\begin{remark}\label{R5.5}
	From  Theorem \eqref{lem2.5}, we have $$\inf_{\Phi_T\in \mathbb{L}^2}J^{\varepsilon}_{Y_1,\omega }(\Phi_T)=-\frac{1}{2}\int_{\omega_T}|\widehat{\varphi}|^2\,dx dt.$$
\end{remark}

Now, we state our main result in this paper, concerning  the cost of interior approximate controllability. In what follows, $C$ stands for a generic positive constant only  depending on $\Omega$  and $\omega$, whose value can change from line to line.

\begin{theorem}\label{t5.6}
	For any target $Y_1  =(y_1, y_{\Gamma,1})\in\mathbb{H}^2$, $T>0$, $\varepsilon>0$, $a\in L^{\infty}(\Omega_T)$, $b\in L^{\infty}(\Gamma_T)$, $B\in L^{\infty}(\Omega_T)^N$ and $B_{\Gamma}\in L^{\infty}(\Gamma_T)^N$ one has
	\begin{equation*}
	\mathcal{C}(0, Y_1,\varepsilon, \omega)\leq \exp\Big(C\big(N(T,a,b,B,B_{\Gamma}) +  \frac{1}{\epsilon} M(a,b,B,B_{\Gamma}, Y_1)  \Big)\|Y_1\|_{\mathbb{L}^2},
	\end{equation*} 
	where 
	\begin{align*}
	N(T, a,b,B,B_{\Gamma})&= 1 +\frac{1}{T}+ T(\|a\|_{\infty} + \|b\|_{\infty} +\|B\|^2_{\infty} +\|B_{\Gamma}\|^2_{\infty})\nonumber\\
	&+ \|a\|^{\frac{2}{3}}_{\infty} + \|b\|^{\frac{2}{3}}_{\infty} + \|B\|^2_{\infty} +\|B_{\Gamma}\|^2_{\infty}
	\end{align*}
	and 
	\begin{align*}
	&M(a,b,B,B_{\Gamma}, Y_1)=  \|a\|_{\infty} + \|b\|_{\infty} +\|B\|^2_{\infty} +\|B_{\Gamma}\|^2_{\infty}+ \nonumber \\
	&(\|a\|_{\infty} + \|b\|_{\infty})\|Y_1\|_{\mathbb{L}^2}+ (\|B\|_{\infty} + \|B_{\Gamma}\|_{\infty})\|Y_1\|_{\mathbb{H}^1} +\|Y_1\|_{\mathbb{H}^2}.
	\end{align*}

\end{theorem}

In order to prove  Theorem \ref{t5.6}, we mainly  need   the following observability  inequality with explicit constants.  Using  some ideas as in  the proof of  {\cite[Theorem 1.1]{B4}}, we deduce our result.

Let us first recall the following observability inequality which is a consequence of Theorem \ref{ThKh2.7} and the same ideas as in \cite{KhMa19}.
\begin{theorem}\label{t5.7}
	Let $a\in L^{\infty}(\Omega_T)$,   $b\in L^{\infty}(\Gamma_T)$, $B\in L^{\infty}(\Omega_T)^N$, $B_{\Gamma}\in L^{\infty}(\Gamma_T)^N$   and $\Phi=(\varphi, \varphi_\Gamma)$ the solution  to \eqref{1.5}. Then, one has
	\begin{equation}\label{5.233}
	\lVert\varphi(0)\lVert^2_{L^2(\Omega)}   +  \lVert\varphi_{\Gamma}(0)\lVert^2_{L^2(\Gamma)}\leq C(T, a,b,B,B_{\Gamma}) \int_{\omega_T}|\varphi|^2 dx dt,     
	\end{equation}
	where
	\begin{align*}
	\small
	C(T, a,b,B,B_{\Gamma})&= \exp\Big(C\big(1+ \frac{1}{T} + T(\lVert a\lVert_{\infty} + \lVert b\lVert_{\infty}+\|B\|_{\infty}+\|B_{\Gamma}\|_{\infty}) \\
	&+\|B\|^2_{\infty}+ \|B_{\Gamma}\|^2_{\infty} + \lVert a\lVert^{\frac{2}{3}}_{\infty}+ \lVert b\lVert^{\frac{2}{3}}_{\infty}\big)\Big).
	\end{align*}
\end{theorem}
\begin{remark}
	(i)\; 	Observe that \eqref{5.233} provides precise estimate on how the observability constant depends on $T$ and size of the parameters $a$, $b$, $B$ and $B_{\Gamma}$. This will be essential when dealing with the semilinear problems.
	
	(ii)\; 
	As proved in {\cite{B8}}, the null controllability result is a consequence of the observability inequality (Theorem \ref{t5.7}), while we can deduce the approximate controllability directly from the Carleman estimate (Lemma \ref{l5.8}).
	
	(iii)\; 
	From the inequality \eqref{5.233}, we deduce a bound  to the  cost of null controllability. Namely, we have, for a positive constant $C>0$ depending on $\omega$,
	\begin{equation*}
	\mathcal{C}^0(Y_0, 0,\omega)\leq C(T, a,b,B,B_{\Gamma})\|Y_0\|_{\mathbb{L}^2},
	\end{equation*}
	where   $$\mathcal{C}^0(Y_0,0, \omega)= \inf_{v\in \mathcal{U}_{ad} (Y_0, 0,\omega)}\|v\|_{L^2(\omega_T)}, \,\,\,\text{and}$$ $$\mathcal{U}_{ad} (Y_0,0, \omega)= \big\{ v\in L^2(\omega_T): y(T)=0,\, y_{\Gamma}(T)=0 \big\}.$$
\end{remark}
\vspace{0.3cm}
\par Now, we are  ready to prove the main result of  this paper:  
\begin{proof}[Proof of Theorem \ref{t5.6}]	Following {\cite{B4}}, let us rewrite $J^{\varepsilon}_{Y_{1}, \omega}$ as follows
	\begin{align*} 
	J^{\varepsilon}_{Y_1,\omega}(\Phi_T) &=   J^{\delta}_{ Y_{1},\omega}(\Phi_T)  +     \varepsilon\lVert \Phi_T\lVert_{\mathbb{L}^2(\Gamma)}  - \big\langle Y_1, \Phi_T-\Phi(T-\delta)\big\rangle_{\mathbb{L}^2},                                                         
	\end{align*}
	where
	\begin{equation*}
	J^{\delta}_{ Y_{1},\omega}(\Phi_T)=  \frac{1}{2}\int_{\omega_{T}}|\varphi|^2\, dx dt \, -\, \big\langle Y_1, \Phi(T-\delta)\big\rangle_{\mathbb{L}^2} ,\,\, \forall \,\,\Phi_T\in\mathbb{L}^2. 
	\end{equation*}
	Here, $\Phi=(\varphi, \varphi_{\Gamma} )$ is the solution to the system adjoint \eqref{1.5}, with final data $\Phi_T.$ 
	The positive number $\delta>0$, small enough, will be fixed later  such that  
	\begin{equation}\label{0.11}
	\varepsilon\lVert \Phi_T\lVert_{\mathbb{L}^2}  - \big\langle Y_1, \Phi_T-\Phi(T-\delta)\big\rangle_{\mathbb{L}^2} \geq 0, \,\, \,\,\Phi_T\in\mathbb{L}^2. 
	\end{equation}
	With this choice of $\delta$, we have 
	\begin{equation}\label{0.12}
	I_1= \inf_{\Phi_{T}\in \mathbb{L}^2} J^{\varepsilon}_{Y_1,\omega}(\Phi_T)\geq\inf_{\Phi_{T}\in \mathbb{L}^{2}} J^{\delta}_{Y_{1},\omega}(\Phi_T)=I_2.
	\end{equation}
	By Remark \ref{R5.5}, we have
	\begin{equation*}
	I_1= \inf_{\Phi_{T}\in \mathbb{L}^2} J^{\varepsilon}_{Y_1,\omega}(\Phi_T)=-\frac{1}{2} \int_{\omega_T}|\widehat{\varphi}|^2 \,\,dx dt, 
	\end{equation*}           
	and hence 
	\begin{equation*}
	\mathcal{C}(0,Y_1,\varepsilon, \omega)^2=  \int_{\omega_T}|\widehat{\varphi}|^2 \,\,dx dt\leq -2I_2 .
	\end{equation*}
	Thus, our task is reduced to estimate $I_2$ and choose $\delta$ such that (\ref{0.11}) holds.
	From the observability inequality \eqref{5.233},
	we have
	\begin{equation}\label{4.42}
	\lVert\varphi(T -\delta)\lVert^2_{L^2(\Omega)}   +  \lVert\varphi_{\Gamma}(T -\delta)\lVert^2_{L^2(\Gamma)}\leq C_{\delta}(a,b,B, B_{\Gamma})  \int_{\omega_T}|\varphi|^2 \,\,dx dt,     
	\end{equation}
	where \begin{align*}
	\ln( C_{\delta}(a,b,B, B_{\Gamma}))&=C\Big(1+ \frac{1}{\delta} + \delta(\lVert a\lVert_{\infty} + \lVert b\lVert_{\infty}\\ 
	&+\|B\|_{\infty}+\|B_{\Gamma}\|_{\infty}) +\|B\|^2_{\infty}+\|B_{\Gamma}\|^2_{\infty}
	+ \lVert a\lVert^{\frac{2}{3}}_{\infty}+ \lVert b\lVert^{\frac{2}{3}}_{\infty}\Big)\\
	&= \tilde{C}_{\delta}(a,b,B, B_{\Gamma}).
	\end{align*}
	Using \eqref{4.42} and the definition of $I_2$, one deduces
	\begin{equation*}
	I_2\geq -\frac{1}{2}  C_{\delta}(a,b,B, B_{\Gamma}) \lVert Y_1\lVert^2_{\mathbb{L}^2}.
	\end{equation*}
	Indeed, using the observability inequality \eqref{4.42}, one has
	{\small
		\begin{align*} 
		\small
		&	J_{\delta, Y_1}(\Phi_T) = \frac{1}{2}\big[\|\varphi\|^2_{L^2(\omega_T)} -2\big\langle Y_1, \Phi(T-\delta)\big\rangle_{\mathbb{L}^2}\big]\\
		&\geq \frac{1}{2}\big[ \exp\big(-\tilde{C}_{\delta}(a,b,B, B_{\Gamma})\big) \|\Phi(T-\delta)\|^2_{\mathbb{L}^2}
		-2\big\langle Y_1, \Phi(T-\delta)\big\rangle_{\mathbb{L}^2}\big] \\
		& =\frac{1}{2} \exp\big(-\tilde{C}_{\delta}(a,b,B, B_{\Gamma})\big) \times\\
		&\quad \quad\Big[\big\|\Phi(T-\delta) - \exp\big(-\tilde{C}_{\delta}(a,b,B, B_{\Gamma}))\big)Y_1\big\|^2_{\mathbb{L}^2}- 
		\exp\big(2C\big(-\tilde{C}_{\delta}(a,b,B, B_{\Gamma})\big)\|Y_1\|^2_{\mathbb{L}^2}\Big]\\
		& \geq - \frac{1}{2}\exp\big(C\big(-\tilde{C}_{\delta}(a,b,B, B_{\Gamma})\big)\|Y_1\|^2_{\mathbb{L}^2}.   
		\end{align*}}          
	Hence, we obtain the following estimate
	\begin{equation}\label{F4.44}
	\mathcal{C}(0,Y_1,\varepsilon, \omega)^2\leq - I_2\leq \exp\big(2\tilde{C}_{\delta}(a,b,B, B_{\Gamma})\big)\lVert Y_1\lVert^2_{\mathbb{L}^2}.  
	\end{equation}
	Now, going back to (\ref{0.11}), we observe that it suffices to look for $\delta >0$ such that
	\begin{equation}
	\Big | \big\langle Y_1, \Phi_T - \Phi(T-\delta)\big\rangle_{\mathbb{L}^2}\Big|\leq\varepsilon\lVert \Phi_T\lVert_{\mathbb{L}^2}.  
	\end{equation}  
	Let us prove that   
	{\small  
		\begin{align}\label{4.46}
		&\big| \big\langle Y_1, \Phi_T - \Phi(T-\delta)\big\rangle_{\mathbb{L}^2}\big|\leq\nonumber\\		
		&  L(\delta, a,b,B, B_{\Gamma})\Big[(\|a\|_{\infty}+\|a\|_{\infty})   \|Y_1\|_{\mathbb{L}^2} +(\|B\|_{\infty}+\|B_{\Gamma}\|_{\infty})   \|Y_1\|_{\mathbb{H}^1}   \lVert \Phi_T\lVert_{\mathbb{L}^2}\Big]\nonumber\\
		& + C\delta\|Y_1\|_{\mathbb{H}^{2}}\lVert \Phi_T\lVert_{\mathbb{L}^2}, 
		\end{align} }
	with 
	{\small 
		\begin{equation*}
		L(\delta, a,b,B, B_{\Gamma})=\frac{e^{\delta(\|a\|_{\infty}+\|b\|_{\infty}+\|B\|^2_{\infty}+\|B_{\Gamma}\|^2_{\infty})}-1}{\|a\|_{\infty}+\|b\|_{\infty}+\|B\|^2_{\infty}+\|B_{\Gamma}\|^2_{\infty}}.
		\end{equation*} }
	To finish, if we assume \eqref{4.46} proved, we choose $\delta$  such that \,\, {\small $$C\delta\|Y_1\|_{\mathbb{H}^{2}} +L(\delta, a,b,B, B_{\Gamma})\Big[(\|a\|_{\infty}+\|a\|_{\infty})   \|Y_1\|_{\mathbb{L}^2} +(\|B\|_{\infty}+\|B_{\Gamma}\|_{\infty})   \|Y_1\|_{\mathbb{H}^1}\Big] \leq\varepsilon,$$}
	and we estimate the quantity $$1+ \frac{1}{\delta} + \delta(\lVert a\lVert_{\infty} + \lVert b\lVert_{\infty}+\|B\|_{\infty}+\|B\|^2_{\infty}+\|B_{\Gamma}\|_{\infty}+\|B_{\Gamma}\|^2_{\infty}).$$
	Let us start from the identity
	\begin{equation}
	\Phi(T-\delta) = e^{\delta A}\Phi_T -\int_{0}^{\delta}e^{(\delta -s) A}F_T(s) ds,
	\end{equation}   
	where 
	{\small 
		\begin{align*}
		&F_T(s)= \Big(a(T-s)\varphi(T-s),b(T-s)\varphi_{\Gamma}(T-s) \Big)\\
		&-\Big(\text{div}(B(T-\delta)\varphi(T-s)),\text{div}_{\Gamma}(B_{\Gamma}(T-s)\varphi_{\Gamma}(T-s))\Big).
		\end{align*}     }               
	With these notations, one has
	{\small 
		\begin{equation}\label{0.133}
		\big\langle Y_1, \Phi_T - \Phi(T-\delta)\big\rangle_{\mathbb{L}^2}=  \big\langle Y_1, \Phi_T - e^{\delta A}\Phi_T \big\rangle_{\mathbb{L}^2} + \big\langle Y_1,  \int_{0}^{\delta}e^{(\delta -s) A}F_T(s) ds\big\rangle_{\mathbb{L}^2}.
		\end{equation} }
	In the sequel,  we estimate the two terms on the right hand side of \eqref{0.133}.  Let us starting by the first term, that is, \,\,$\big\langle Y_1, \Phi_T - e^{\delta A}\Phi_T \big\rangle_{\mathbb{L}^2}.$
	We have 
	\begin{equation*}
	\big\langle Y_1, \Phi_T - e^{\delta A}\Phi_T \big\rangle_{\mathbb{L}^2}\leq \lVert Y_1\lVert_{\mathbb{H}^2}\lVert \Phi_T - e^{\delta A}\Phi_T \lVert_{\mathbb{H}^{-2}}.
	\end{equation*} 
	Here, $\mathbb{H}^{-2}$ designates the dual space  of $\mathbb{H}^{2}$ with respect to the pivot space $\mathbb{L}^2$.
	In order to estimate $ \lVert \Phi_T - e^{\delta A}\Phi_T\lVert_{\mathbb{H}^{-2}}$, let us introduce the following homogeneous Cauchy problem with initial data $\Phi_T$, that is,
	\begin{equation}\label{4.51}
	\left\{
	\begin{array}{ll}
	X'(t) = AX(t)\quad  &t\in (0,T),  \\
	X(0)=\Phi_T\in \mathbb{L}^2  .
	\end{array}
	\right.
	\end{equation}
	We have                         
	\begin{align*}
	\Vert \Phi_T - e^{\delta A}\Phi_T \Vert_{\mathbb{H}^{-2}} & =\left\Vert \int_{0}^{\delta} X'(t) dt \right\Vert_{\mathbb{H}^{-2}}\\
	& \leq \int_{0}^{\delta}  \left\Vert  AX(t)\right\Vert_{\mathbb{H}^{-2}}  dt\\
	&\leq C \int_{0}^{\delta}  \Vert X(t) \Vert_{\mathbb{L}^{2}}dt\leq C\delta \max_{0\leq t\leq \delta}\lVert X(t)\lVert_{\mathbb{L}^{2}}. 
	\end{align*}
	Multiplying \eqref{4.51}  by $X$ and integrating, we have
	\begin{equation*}
	\frac{1}{2} (\|X(t)\|^2_{\mathbb{L}^2}- \|X(0)\|^2_{\mathbb{L}^2})=\big\langle AX(t),X(t)\big\rangle_{\mathbb{L}^{2}}
	\end{equation*}  
	for all $0\leq t\leq\delta$.	Since $A$ is negative, one deduces  
	\begin{equation*}         
	\max_{0\leq t\leq \delta}\lVert X(t)\lVert_{\mathbb{L}^{2}}  \leq \lVert \Phi_T\lVert_{\mathbb{L}^{2}}, 
	\end{equation*}   
	and 	hence, 
	\begin{equation*}         
	\lVert \Phi_T - e^{\delta A}\Phi_T \lVert_{\mathbb{H}^{-2}}\leq C\delta \lVert \Phi_T\lVert_{\mathbb{L}^{2}}.
	\end{equation*} 
	We have then, as a first conclusion
	\begin{equation} \label{F2.48}        
	\big\langle Y_1, \Phi_T - e^{\delta A}\Phi_T \big\rangle_{\mathbb{L}^2}\leq C\delta \lVert Y_1\lVert_{\mathbb{H}^{2}} \lVert \Phi_T\lVert_{\mathbb{L}^{2}}.
	\end{equation} 
	\vspace{0.1 cm} 
	Now, we shall estimate the term  $\displaystyle\big\langle Y_1, \int_{0}^{\delta}e^{(\delta -s) A}F_T(s) ds \big\rangle_{\mathbb{L}^2}$. 
	We know that the operator $A$ is a negative and self-adjoint then, by means of Lumer-Phillips Theorem, it generates a $C_0$-semigroup of contraction. Then, we have
	\begin{align}         
	\displaystyle\big\langle Y_1, \int_{0}^{\delta}e^{(\delta -s) A}F_T(s)  \big\rangle_{\mathbb{L}^2} ds &= \int_{0}^{\delta}\big\langle Y_1, e^{(\delta -s) A}F_T(s)  \big\rangle_{\mathbb{L}^2} ds \nonumber\\
	&\leq(\|a\|_{\infty} +\|b\|_{\infty}  )\lVert Y_1\lVert_{\mathbb{L}^{2}} \int_{0}^{\delta}\lVert \Phi(T-s) \lVert_{\mathbb{L}^{2}} ds\nonumber \\
	& +  (\|B\|_{\infty}  +\|B_{\Gamma}\|_{\infty})\lVert Y_1\lVert_{\mathbb{H}^{1}}  \int_{0}^{\delta}\lVert \Phi(T-s) \lVert_{\mathbb{L}^{2}} ds.\label{F2.49}
	\end{align}          
	On the other hand, multiplying  the adjoint system by $\Phi=(\varphi, \varphi_{\Gamma})$ and integrating, we get 
	\begin{align*}
	&-\frac{1}{2}\frac{d}{dt}(\|\Phi(t)\|_{\mathbb{L}^2}^2) +   \int_{\Omega}\mathcal{A}\nabla\varphi\cdot\nabla \varphi dx \, -  \int_{\Omega}\varphi B \cdot\nabla \varphi dx+ \int_{\Omega} a |\varphi|^2 dx \\
	& + \int_{\Gamma}\mathcal{A}_{\Gamma}\nabla_\Gamma \varphi_\Gamma\cdot \nabla_\Gamma \varphi_\Gamma d\sigma \,-\int_{\Omega}\varphi B_{\Gamma} \cdot\nabla_{\Gamma} \varphi_{\Gamma} d\sigma  + \int_{\Gamma} b |\varphi_\Gamma|^2 d\sigma =0.
	\end{align*}
	By young inequality, and for nonnegative $\lambda$, we have 
	\begin{equation*}
	\int_{\Omega}\varphi B \cdot\nabla \varphi dx\leq \frac{\lambda}{2}\|B\|^2_{\infty}\|  \varphi\|^2_{L^2(\Omega)}  + \frac{1}{2\lambda}\| \nabla\varphi\|^2_{L^2(\Omega)},
	\end{equation*}
	\begin{equation*}
	\int_{\Gamma}\varphi B_{\Gamma} \cdot\nabla_{\Gamma} \varphi_{\Gamma} dx\leq \frac{\lambda}{2}\|B_{\Gamma}\|^2_{\infty}\|  \varphi_{\Gamma}\|^2_{L^2(\Gamma)}  + \frac{1}{2\lambda}\| \nabla_{\Gamma}\varphi_{\Gamma}\|^2_{L^2(\Gamma)}.
	\end{equation*}	
	Using the ellipticity of $\mathcal{A}$ and $\mathcal{A}_{\Gamma} $, and choosing $\lambda$ large enough,  we have  
	\begin{align*}
	-\frac{1}{2}\frac{d}{dt}(\|\Phi(t)\|_{\mathbb{L}^2}^2) \leq C(\|a\|_{\infty}+\|b\|_{\infty}+\|B\|^2_{\infty}+\|B_{\Gamma}\|^2_{\infty})\|\Phi(t)\|_{\mathbb{L}^2}^2.
	\end{align*}
	This proves in particular that the real function
	\begin{equation*}         
	t\longmapsto e^{t(\|a\|_{\infty}+\|b\|_{\infty}+\|B\|^2_{\infty}+\|B_{\Gamma}\|^2_{\infty})}\|\Phi(t)\|_{\mathbb{L}^2}^2
	\end{equation*}  
	is increasing. Therefore, 
	\begin{align*}         
	\int_{0}^{\delta}\|\Phi(T-s)\|_{\mathbb{L}^2}\,ds&\leq\int_{0}^{\delta} e^{s(\|a\|_{\infty}+\|b\|_{\infty}+\|B\|^2_{\infty}+\|B_{\Gamma}\|^2_{\infty})}\,ds\|\Phi_T\|_{\mathbb{L}^2}\\
	&=   \frac{e^{\delta(\|a\|_{\infty}+\|b\|_{\infty}+\|B\|^2_{\infty}+\|B_{\Gamma}\|^2_{\infty})}-1}{\|a\|_{\infty}+\|b\|_{\infty}+\|B\|^2_{\infty}+\|B_{\Gamma}\|^2_{\infty}}\|\Phi_T\|_{\mathbb{L}^2}.
	\end{align*}
	From the last inequality, \eqref{F2.48} and \eqref{F2.49} we  deduce \eqref{4.46}.
	It is suffices to take $\delta$ such that 
	{\small 
		\begin{equation*}
		C\delta\|Y_1\|_{\mathbb{H}^{2}} +     L(a,b,B, B_{\Gamma}) \Big[(\|a\|_{\infty} +  \|b\|_{\infty} )\|Y_1\|_{\mathbb{L}^2}  +  (\|B\|_{\infty} +  \|B_{\Gamma}\|_{\infty} ) \|Y_1\|_{\mathbb{H}^1}\Big]\leq \varepsilon.
		\end{equation*} }
	Choosing  
	{\small
		\begin{align*}
		\delta=\min\left\{ T , \frac{\varepsilon}{3C\|Y_1\|_{\mathbb{H}^{2}}}, K_1 , K_2\right\}, 
		\end{align*}}
	where
	{\small
		\begin{align*}\label{4.63}
		K_1 &= \frac{1}{\|a\|_{\infty} + \|b\|_{\infty} +\|B\|^2_{\infty}+\|B_{\Gamma}\|^2_{\infty}} \ln(1+ \frac{\varepsilon}{3(\|a\|_{\infty} + \|b\|_{\infty})\|Y_1\|_{\mathbb{L}^{2}}})\\
		K_2 &= \frac{1}{\|a\|_{\infty} + \|b\|_{\infty} +\|B\|^2_{\infty}+\|B_{\Gamma}\|^2_{\infty}} \ln(1+ \frac{\varepsilon}{3(\|B\|^2_{\infty} + \|B_{\Gamma}\|^2_{\infty})\|Y_1\|_{\mathbb{H}^{1}}}),
		\end{align*}}
	we can easily show that 
	{\small
		\begin{align*}
		&1\, +\,\frac{1}{\delta}\, + \,\delta\,(\,\|a\|_{\infty}\, +\, \|b\|_{\infty}+ \|B\|^2_{\infty}+\|B_{\Gamma}\|^2_{\infty})\\
		&\leq C\big(1 \,+ \,\frac{1}{T}\, + \,T(\|a\|_{\infty}\, +\, \|b\|_{\infty}+ \|B\|^2_{\infty}+\|B_{\Gamma}\|^2_{\infty}) \,\,+\\
		&\,\,\frac{1}{\varepsilon}\Big(\|Y_1\|_{\mathbb{H}^{2}} + (\|a\|_{\infty} + \|b\|_{\infty})\|Y_1\|_{\mathbb{L}^{2}} +(\|B\|^2_{\infty} + \|B_{\Gamma}\|^2_{\infty})\|Y_1\|_{\mathbb{H}^{1}}\\
		&+ \|a\|_{\infty}+\|b\|_{\infty}+\|B\|^2_{\infty}+\|B_{\Gamma}\|^2_{\infty}\Big).
		\end{align*} }

	Finally, using this inequality  and  \eqref{F4.44}, we deduce  the desired inequality.  \qed
\end{proof} 
\begin{remark}
	It is important to notice that Theorem \ref{t5.6} guarantees that $v$
	remains uniformly bounded when the potentials $a$, $b$, $B$ and $B_{\Gamma}$ remain bounded in $L^{\infty}(\Omega)$, $L^{\infty}(\Gamma_T)$, $(L^{\infty}(\Omega))^N$ and $(L^{\infty}(\Gamma_T))^N$, respectively. This shall play  a crucial role when dealing with the semilinear case.
\end{remark}

\section{ The  semilinear case }\label{sss.3}
In this section we deduce the cost of approximate controllability of  semilinear heat equation with dynamic boundary conditions as a consequence of the results obtained in the linear case and a fixed point technique described  by E. Zuazua in \cite{Zucnl}.
Let us  consider the following system
\begin{equation}
\left\{\begin{array}{ll}
{\partial_t y-\text{div}(\mathcal{A}\nabla y)  + F (y,\nabla y)= 1_{\omega}v }& {\text { in } \Omega_T,} \\
{\partial_t y_{_{\Gamma}}   - \text{div}_{\Gamma}(\mathcal{A}_{\Gamma}\nabla_{\Gamma} y_{\Gamma})    + \partial_{\nu}^{\mathcal{A}}y+ G (y_{\Gamma},\nabla_{\Gamma} y_{\Gamma})=0} &\,\,{\text {on} \,\Gamma_T,} \\
y_{\Gamma}(x,t) = y\rvert_{\Gamma}(x,t) &\,\,\text{on } \Gamma_T, \\
{(y(0),y_{\Gamma}(0)) =(y_0,y_{\Gamma,0})} & {\text { in } \Omega \times \Gamma}, \label{nl2.62}
\end{array}\right.
\end{equation}
with  $F,G\in C^1(\mathbb{R})$ such that $F(0)=G(0)=0$. We assume furthermore that there exist $L_F , L_G>0$ such that
\begin{align}
|F(x,\zeta)-F(y,\xi)|&\leq L_F(|x-y| + |\zeta- \xi|)\label{Lip2.63}\\
|G(x,\zeta)-G(y,\xi)|&\leq L_G(|x-y| + |\zeta- \xi|) \label{Lip2.64}
\end{align}
for all \,\,$x, y \in \mathbb{R}$\, and \,\,$\zeta, \xi \in \mathbb{R}^N$.  Recall that for $Y_0=(y_0, y_{\Gamma,0})\in \mathbb{L}^2$ and $v\in L^2(\omega_T)$, under assumptions \eqref{Lip2.63}-\eqref{Lip2.64},  \eqref{nl2.62} has a unique weak  solution which  belongs to $C([0, T]; \mathbb{L}^2) \cap L^2(0, T; \mathbb{H}^1)$. For some results on controllability issue for nonlinear static heat equation, we refer, for instance, to  \cite{Ba02, Ba00, Ba99}. See also the monograph \cite{Cor12}. 

Following the description given by Zuazua and Fabre et al.  in  \cite{FPZ'93} and \cite{Zucnl}, respectively, the system \eqref{nl2.62} can be rewritten as 
{\small
	\begin{equation}
	\left\{\begin{array}{ll}
	{\partial_t y-\text{div}(\mathcal{A}\nabla y) + \tilde{ F}_1(y)y +\tilde{ F}_2(y)\cdot \nabla y= 1_{\omega}v }& {\text { in } \Omega_T,} \\
	{\partial_t y_{_{\Gamma}}   -\text{div}_{\Gamma}(\mathcal{A}_{\Gamma}\nabla_{\Gamma} y_{\Gamma})   + \partial^{\mathcal{A}}_{\nu}y + \tilde{ G}_1(y_{\Gamma})y_\Gamma +\tilde{ G}_2(y)\cdot \nabla_{\Gamma}y_{\Gamma}=0} &\,\,{\text {on} \,\Gamma_T,} \\
	y_{\Gamma}(x,t) = y\rvert_{\Gamma}(x,t) &\,\,\text{on } \Gamma_T, \\
	{(y(0),y_{\Gamma}(0)) =(y_0,y_{\Gamma,0})} & {\text { in } \Omega \times \Gamma} \label{1.72}, 
	\end{array}\right.
	\end{equation}}
where $\tilde{ F}_1$, $\tilde{ F}_2$, $\tilde{ G}_1$ and $\tilde{ G}_2$ are  real functions satisfying
\begin{align*}
\|\tilde{ F}_1(y)\|_{\infty}&\leq L_F\,\, \text{and}\,\, \|\tilde{ G}_1(y_\Gamma)\|_{\infty}\leq L_G,\\
\|\tilde{ F}_2(y)\|_{\infty}&\leq L_F\,\, \text{and}\,\, \|\tilde{ G}_2(y_\Gamma)\|_{\infty}\leq L_G
\end{align*}
for all  $y\in L^2(0,T; H^1(\Omega))$ and\, $y_\Gamma\in L^2(0,T; H^1(\Gamma))$, see \cite{Zucnl} for the explicit expression of these functions.

Fix $\overline{Y}=(\overline{y},\overline{y}_\Gamma)$ $\in L^2(0, T; \mathbb{H}^1)$, $Y_1\in \mathbb{L}^2$, $\varepsilon>0$, and denote  $a=\tilde{ F}_1(\overline{y})$, $b=\tilde{ G}_1(\overline{y}_\Gamma)$, $B=\tilde{ F}_2(\overline{y}_\Gamma)$ and $B_{\Gamma}=\tilde{ G}_2(\overline{y}_\Gamma)$.   The results of the previous sections  above (the linear case) show that there exists a unique  control $v=v(\overline{Y})$, control with minimal norm, such that the solution $ Y=(y, y_\Gamma)$ of the  following  linearized  system 
{\small
	\begin{equation}
	\left\{\begin{array}{ll}
	{\partial_t y-\text{div}(\mathcal{A}\nabla y)  +B(x,t)\cdot \nabla y+ a(x,t)y= 1_{\omega}v }& {\text { in } \Omega_T,} \\
	{\partial_t y_{_{\Gamma}}   - \text{div}_{\Gamma}(\mathcal{A}_{\Gamma}\nabla_{\Gamma} y_{\Gamma})    + \partial_{\nu}^{\mathcal{A}}y+B_{\Gamma}(x, t)\cdot \nabla_{\Gamma} y_{\Gamma} + b(x,t)y_{\Gamma}=0} &\,\,{\text {on} \,\Gamma_T,} \\
	y_{\Gamma}(x,t) = y\rvert_{\Gamma}(x,t) &\,\,\text{on } \Gamma_T, \\
	{(y(0),y_{\Gamma}(0)) =(y_0,y_{\Gamma,0})} & {\text { in } \Omega\times \Gamma} \label{nl4.744}
	\end{array}\right.
	\end{equation}}
satisfies the approximate controllability property, i.e.,
\begin{equation}
\| Y(T, \overline{Y} )- Y_1\|_{\mathbb{L}^2}\leq \varepsilon.
\end{equation}
To deduce  the approximate controllability result for \eqref{nl2.62}, it suffices to prove that the nonlinear mapping
\begin{align*}
\mathcal{N}: L^2(0, T; \mathbb{H}&^1)\longrightarrow L^2(0, T; \mathbb{H}^1)\\
& \overline{Y}\longmapsto Y,\nonumber 
\end{align*}
where  $Y = \mathcal{N}(\overline{Y})$ is the solution to \eqref{nl4.744} associated to $v=v(\overline{Y})$, $a=\tilde{ F}(\overline{y})$ and $b=\tilde{ G}(\overline{y}_\Gamma)$, admits a fixed point. Indeed, if  $\hat{Y}$ is a fixed point of $\mathcal{N}$, the control $v(\hat{Y})$ is one we are looking for, that is $v(\hat{Y})$ is a control for the semilinear problem \eqref{nl2.62}.
\par In order to study this semilinear case, we need to make precise the dependence of the solution $\Phi$ of the adjoint system \eqref{1.5} and the minimum $\hat{\Phi}_T$ of $J^{\varepsilon}_{Y_1}(\cdot)$ with parameters $a$, $b$, $B$ and $B_{\Gamma}$. To this end, let us introduce the functionals
\begin{align}
\mathcal{L}: \mathbb{L}^{\infty}_T\times(\mathbb{L}^{\infty}_T)^N\times\mathbb{L}^2\longrightarrow L^2(0, T; \mathbb{L}^2); \big((a,b),(B, B_{\Gamma}), \Phi_T\big)\longmapsto \Phi,  
\end{align} 
and 
\begin{align}
\mathcal{M}: \mathbb{L}^{\infty}_T\times (\mathbb{L}^{\infty}_T)^N\longrightarrow \mathbb{L}^2;
((a,b), (B, B_{\Gamma}))\longmapsto \hat{\Phi}_T,
\end{align}
where  $\mathbb{L}^{\infty}_T=L^{\infty}(\Omega_T)\times L^{\infty}(\Gamma_T)$.
The functionals  $\mathcal{L}$ and  $\mathcal{M}$ satisfy the following properties. 
\begin{proposition}\label{SEM.1}
	Let $((a_n, b_n), (B_n, B_{\Gamma, n})) $ be a sequence in $\mathbb{L}^{\infty}_T\times (\mathbb{L}^{\infty}_T)^N $, and $(\Phi_{T,n})$ a sequence in $\mathbb{L}^2.$
	\begin{itemize} 
		\item[(a)]  If $((a_n, b_n), (B_n, B_{\Gamma, n})) $ is bounded in $\mathbb{L}^{\infty}_T\times (\mathbb{L}^{\infty}_T)^N $, then {\small$(\hat{\Phi}_{T, n})=(\mathcal{M}(a_n, b_n) )$} is bounded in $\mathbb{L}^2$.
		\item[(b)] If $(a_n, b_n) \rightharpoonup (a,b) $ weakly $-*$ in $\mathbb{L}^{\infty}_T$,  $(B_n, B_{\Gamma, n}) \rightharpoonup (B, B_{\Gamma}) $ weakly $-*$ in $(\mathbb{L}^{\infty}_T)^N$ and $\Phi_{T,n} \rightharpoonup  \Phi_{T} $ weakly in $\mathbb{L}^{2}$ then, 
		$\mathcal{L}((a_n,b_n), (B_n, B_{\Gamma, n}); \Phi_{T,n} ) \longrightarrow \mathcal{L}((a,b),(B, B_{\Gamma}); \Phi_T)$  strongly in $L^2(0,T;\mathbb{L}^2 ).$
	\end{itemize} 
	
\end{proposition}
\begin{proof}
	For the  first point, assume that there is a subsequence $\hat{\Phi}^n_T$ such that  $\|\hat{\Phi}^n_T\|_{\mathbb{L}^2}\longrightarrow \infty$. Since $J_{\varepsilon}(\cdot)$ is coercive, we have $J_\varepsilon(\hat{\Phi}^n_T)\longrightarrow \infty$. This contradicts the fact that $J_\varepsilon(\hat{\Phi}^n_T)= -\frac{1}{2}\|\hat{v}_n\|^2_{L^2(\omega_T)} $, which is, viewed the boundness of  $a_n$, $b_n$ $B_n$ and $B_{\Gamma, n}$,  bounded.
	To prove the second point, denote $\Phi_n= \mathcal{L}((a_n,b_n), (B_n, B_{\Gamma, n}); \Phi_{T,n})$ and   $\Phi= \mathcal{L}((a,b), (B, B_{\Gamma});\Phi_{T})$. Using the estimation \eqref{est2.10} in Proposition \ref{prop2.3}  we deduce that 
	$(\Phi_n)$ is bounded in $L^2(0,T; \mathbb{H}^1)$ and $(\Phi^{'}_n)$ is bounded in $L^2(0,T; \mathbb{H}^{-1})$. By Aubin-Lions Lemma, see e.g. {\cite[Corollary 4]{Sim87}}, there exists $\Psi \in L^2(0,T; \mathbb{L}^2)$ such that 
	\begin{equation*}
	\Phi_n \longrightarrow\,   \Psi\,\,\, \text{strongly in} \,\, L^2(0,T; \mathbb{L}^2).
	\end{equation*}
	We pass to limit in the definition of the weak solution, and using some standard arguments,  we obtain that $\Psi=\Phi.$   \qed
\end{proof} 
Once this dependence is understood, we can prove that $\mathcal{N}$ admits a fixed point. To do this, it suffices to use  the  Schauder's fixed point Theorem, and  the following lemma.
\begin{lemma}\label{pfix.1}
	The functional $\mathcal{N}$ satisfies the following:
	\begin{itemize}
		\item[(a)] $\mathcal{N}$ is continuous and compact,
		\item[(b)] there is $R>0$ such that $\|\mathcal{N}(\overline{Y})\|_{L^2(0, T; \mathbb{H}^1)}\leq R$  for all $\overline{Y} \in L^2(0, T; \mathbb{H}^1).$
	\end{itemize}
\end{lemma} 
\begin{proof}
	Continuity of \,\,$\mathcal{N}:$  Take a sequence  $\overline{Y}_n =(\overline{y}_n, \overline{y}_{\Gamma,n})\longrightarrow  \overline{Y}=(\overline{y}, \overline{y}_{\Gamma})$ in \,\,$L^2(0, T; \mathbb{H}^1)$. Then the parameters  
	$\tilde{F}_1(\overline{y}_n)$, $\tilde{G}_1(\overline{y}_{\Gamma,n})$ $\tilde{F}_2(\overline{y}_n)$ and  $\tilde{G}_2(\overline{y}_{\Gamma,n})$ are such that
	\begin{align}
	&\tilde{F}_1(\overline{y}_{n})\longrightarrow \tilde{F}_1(\overline{y}) \,\,\text{in} \,\,L^2(\Omega_T), \,\, \text{and} \,\, \tilde{F}_2(\overline{y}_{n})\longrightarrow \tilde{F}_2(\overline{y}) \,\,\text{in} \,\,L^2(\Omega_T)\label{cv2.71},\\
	&\tilde{G}_1(\overline{y}_{n})\longrightarrow \tilde{G}_1(\overline{y}) \,\,\text{in} \,\,L^2(\Gamma_T), \,\, \text{and} \,\, \tilde{G}_2(\overline{y}_{n})\longrightarrow \tilde{G}_2(\overline{y}) \,\,\text{in} \,\,L^2(\Gamma_T)\label{cv2.72}.
	\end{align}
	
	Denote $a_n=\tilde{F}_1(\overline{y}_n)$, $b_n= \tilde{G}_1(\overline{y}_{\Gamma,n})$, $B_n=\tilde{F}_2(\overline{y}_n)$ and  $B_{\Gamma}=\tilde{F}_2(\overline{y}_n)$. One has
	\begin{equation}\label{nl4.788}
	\|a_n\|_{\infty}\leq L_F, \|B_n\|_{\infty}\leq L_F \,\, \|b_n\|_{\infty}\leq L_G \,\,\text{and} \,\,\|B_{n,\Gamma}\|_{\infty}\leq L_G, \quad \forall\, n\geq 0.
	\end{equation}
	From  Theorem \ref{t5.6} and \eqref{nl4.788}, we deduce that the corresponding controls $\hat{v}_n$ are uniformly bounded
	\begin{equation}\label{nl4.78}
	\|\hat{v}_n\|_{L^2(\omega_T)} \leq C,\,\,\,  \forall n\geq 1.
	\end{equation}
	Recall that $\hat{v}_n= \hat{\varphi}_n$, where $\hat{\Phi}_n=(\hat{\varphi}_n, \hat{\varphi}_{\Gamma,n})$
	is the solution of adjoint system  associated to  $a_n$, $b_n$, $B_n$,  $B_{\Gamma,n}$, and $\hat{\Phi}_{T,n}$,  the minimum of the functional 
	\begin{equation*}
	J_n(\Phi_{T})=J(a_n, b_n, B_n, B_{\Gamma,n}; \Phi_{T})= \frac{1}{2}\|\varphi_n\|^2_{L^2(\omega_T)} + \varepsilon\|\Phi_{T}\|_{\mathbb{L}^2}- \langle Y_1, \Phi_{T}\rangle_{\mathbb{L}^2}.
	\end{equation*}
	Here,  $\Phi_n=(\varphi_n, \varphi_{\Gamma,n})$
	is the solution of adjoint system with final data $\Phi_{T}$ and potentials $a_n$, $b_n$, $B_n$  and $B_{\Gamma,n}$.  From the first point of Proposition \ref{SEM.1}, it follows that
	\begin{equation}\label{nl4.80}
	\|\hat{\Phi}_{T,n}\|_{\mathbb{L}^2} \leq C,\,\,\,  \forall n\geq 1.
	\end{equation} 
	By extracting a subsequence, there exists $\hat{\Phi}_{T}\in\mathbb{L}^2$ such that
	\begin{equation}\label{nl4.822}
	\hat{\Phi}_{T,n}\rightharpoonup \hat{\Phi}_{T}\,\, \text{weakly in}\,\,\mathbb{L}^2.
	\end{equation}
	Using \eqref{cv2.71}, \eqref{cv2.72},  \eqref{nl4.822}  and the second point of Proposition \ref{SEM.1},
	\begin{equation}\label{nl4.83}
	\hat{\Phi}_{n}\longrightarrow \hat{\Phi}_{}\,\, \text{ strongly in}\,\,L^2(0,T;\mathbb{L}^2 ),
	\end{equation}
	where $\hat{\Phi}$ is the solution of the adjoint system associated to $\hat{\Phi}_{T}$, $a=\tilde{F}_1(\overline{y})$, $b=\tilde{G}_1(\overline{y}_\Gamma)$, $B=\tilde{F}_2(\overline{y})$ and $B=\tilde{G}_2(\overline{y})$. 
	In particular, we have
	\begin{equation}\label{nl4.994}
	\hat{v}_{n}\longrightarrow \hat{\varphi}\,\, \text{ strongly in}\,\,L^2(\omega_T ).
	\end{equation} 
	Let us show that $\hat{v}=\hat{\varphi}$ is the control with minimal norm associated to $ \overline{Y}=(\overline{y}, \overline{y}_{\Gamma}).$  To this end it suffices to prove that $\hat{\Phi}_{T}$ is the minimum of $J(a, b, B, B_{\Gamma};\cdot)$.\\
	But this is not difficult. Indeed, from \eqref{nl4.822} and \eqref{nl4.83},  we have 
	\begin{equation}\label{nl4.85}
	J(a, b, B, B_{\Gamma}; \hat{\Phi}_T)\leq \liminf_{n\rightarrow \infty} J(a_n, b_n, B_n, B_{\Gamma,n};\hat{\Phi}_{T,n}),
	\end{equation}
	and
	\begin{equation}\label{nl4.86}
	J(a, b, B, B_{\Gamma};\Psi_T)= \liminf_{n\rightarrow \infty} J(a_n, b_n, B_n, B_{\Gamma, n}; \Psi_{T}) \,\,\text{for all}\,\, \Psi_T\in  \mathbb{L}^2.
	\end{equation}
	By definition, one has
	\begin{equation}\label{nl4.87}
	J(a_n, b_n, B_n, B_{\Gamma,n};\hat{\Phi}_{n,T})\leq J(a_n, b_n, B_n, B_{\Gamma,n};\Psi_{T})\,\,\text{for all}\,\, \Psi_T\in  \mathbb{L}^2.
	\end{equation} 
	By \eqref{nl4.85}, \eqref{nl4.86} and \eqref{nl4.87}, we get
	$$  J(a, b, B, B_{\Gamma};\hat{\Phi}_{T})\leq J(a, b, B, B_{\Gamma};\Psi_{T}) \,\,\text{for all}\,\, \Psi_T\in  \mathbb{L}^2.  $$ 
	Finally, \eqref{nl4.994} implies
	\begin{equation*}
	Y(\hat{v}_{n})\longrightarrow Y(\hat{v})\,\, \text{ strongly in} \,\,L^2(0, T; \mathbb{H}^1).
	\end{equation*}
	The proof of continuity of $\mathcal{N}$ is then completed.  To show the 	compactness of $\mathcal{N}$, 
	let $B$ be a bounded set of $L^2(0, T; \mathbb{H}^1)$.  From what precede,  it is easy to see that $\mathcal{N}(B)\subset L^2(0, T; \mathbb{H}^1)$. We are going to show that $\mathcal{N}(B)$ is relatively compact in $L^2(0, T; \mathbb{H}^1)$.
	To this end, let us rewrite the solution of \eqref{nl4.744} as 
	$$Y= P - Q,$$
	where $P(t)= e^{tA}Y_0$ and $\displaystyle Q(t) =\int_{0}^{t}e^{(t-s)A} H_{\overline{Y}}(s)ds,$
	with $$H_{\overline{Y}}=(-1_\omega v + \tilde{ F}_1(\overline{y})y +\tilde{ F}_2(\overline{y})\cdot \nabla y, \tilde{ G}_1(\overline{y}_\Gamma)y_\Gamma +\tilde{ G}_2(\overline{y_{\Gamma}})\cdot \nabla_{\Gamma} y_{\Gamma}),$$
	which is uniformly bounded in $L^2(0, T; \mathbb{H}^1)$.
	Obviously $P$ is a fixed element in  $L^2(0, T; \mathbb{H}^1)$, and by the analyticity of the semigroup $(e^{tA})_{t\geq 0}$ and {\cite[Theorem 2.3]{Di'84}} we get that $Q$ lies in a bounded set of $\mathbb{E}_1$ which is  a relatively compact set in  $L^2(0, T; \mathbb{H}^1)$. This completes the proof of the compactness of $\mathcal{N}$. Show the 
	boundedness of the range of $\mathcal{N}$. 
	Theorem \ref{t5.6} shows that there exists $C>0$ such that 
	\begin{equation}
	\|v(\overline{Y})\|_{L^2(\omega_T)}\leq C.
	\end{equation}
	Energy estimate \eqref{4.3} and this last inequality show  that
	$$\|\mathcal{N}(\overline{Y})\|_{L^2(0, T; \mathbb{H}^1)} \leq C,$$
	where the constant $C$  depends only
	on the Lipschitz constants $L_F$, $L_G$ and the initial data $Y_0$, but  independent on $\overline{Y}.$  \qed
\end{proof}

In view of results  of Lemma \ref{pfix.1} and as a consequence of Schauder's fixed point theorem, we deduce the existence of a fixed point of $\mathcal{N}$. Namely, we have the following result.
\begin{corollary}
	The nonlinear functional $\mathcal{N}$ admits a fixed point, and the system \eqref{nl2.62} is approximately controllable.
\end{corollary}
We are now ready to deal with the cost of approximate controllability issue. Consider, as in the linear case, the admissible set
\begin{align*}
&\mathcal{V}_{ad} (0,Y_1,\varepsilon,\omega)=\\
&\big\{ v\in L^2(\omega_T) :   \text{the solution}\,\,  Y \text{of} \,\,\eqref{nl2.62}\,\,\text{with}\, Y_0=0\,\text{satisfies}\,\,\,\|Y(T)-Y_1\|_{\mathbb{L}^2}\leq \varepsilon\big\},
\end{align*}
and the cost of approximate controllability
$$\mathcal{C}_1(0,Y_1,\varepsilon, \omega)= \inf_{v\in \mathcal{V}_{ad} (0,Y_1,\varepsilon,\omega)}\|v\|_{L^2(\omega_T)}.$$
By using Theorem \ref{t5.6}, \eqref{Lip2.63} and \eqref{Lip2.64} we can deduce  the following result.

\begin{theorem}
	For any target $Y_1  =(y_1, y_{\Gamma,1})\in\mathbb{H}^2$, $T>0$, $\varepsilon>0$, F and G satisfying  \eqref{Lip2.63} and \eqref{Lip2.63}, one has
	\small
	\begin{equation*}
	\mathcal{C}_1(0, Y_1,\varepsilon, \omega)\leq \exp\Big(C\big(N_1(T, L_F, L_G) +  \frac{1}{\epsilon} M_1(L_F, L_G, Y_1)  \Big)\|Y_1\|_{\mathbb{L}^2},
	\end{equation*}
	where 
	\begin{align*}
	N_1(T, L_F, L_G)&= 1 +\frac{1}{T}+ T(\|L_F\|_{L^{\infty}} + \|L_G\|_{L^{\infty}} +\|L_F\|^2_{L^{\infty}} +\|L_G\|^2_{L^{\infty}})\\
	&+ \|L_F\|^{\frac{2}{3}}_{L^{\infty}} +\|L_G\|^{\frac{2}{3}}_{L^{\infty}}+\|L_F\|^2_{L^{\infty}} +\|L_G\|^2_{L^{\infty}}
	\end{align*}
	and 
	\begin{align*}
	M_1(L_F, L_G, Y_1)&=  \|L_F\|_{L^{\infty}} + \|L_G\|_{L^{\infty}} +\|L_F\|^2_{L^{\infty}} +\|L_G\|^2_{L^{\infty}}+ \\
	&(\|L_F\|_{L^{\infty}} + \|L_G\|_{L^{\infty}})\|Y_1\|_{\mathbb{L}^2}+ (\|L_F\|_{L^{\infty}} + \|L_G\|_{L^{\infty}})\|Y_1\|_{\mathbb{H}^1} +\|Y_1\|_{\mathbb{H}^2}.
	\end{align*}
\end{theorem}


\end{document}